\documentclass{article}
\usepackage{amsmath}
\usepackage{amssymb}
\usepackage{amsfonts}
\usepackage{amscd}
\usepackage{mathrsfs}
\usepackage{amsthm}
\usepackage{natbib}

\newcommand{\CC}{{\rm\bf C}}
\newcommand{\RR}{{\rm\bf R}}
\newcommand{\QQ}{{\rm\bf Q}}
\newcommand{\ZZ}{{\rm\bf Z}}

\DeclareMathOperator{\SO}{\mathrm {SO}}

\DeclareMathOperator{\Hom}{\mathrm {Hom}}
\DeclareMathOperator{\Ext}{\mathrm {Ext}}

\DeclareMathOperator{\kernel}{\mathrm {ker}}

\newcommand{\liea}{{\mathfrak {a}}}

\newcommand{\lieg}{{\mathfrak {g}}}
\newcommand{\lieh}{{\mathfrak {h}}}

\newcommand{\liek}{{\mathfrak {k}}}
\newcommand{\liel}{{\mathfrak {l}}}

\newcommand{\lien}{{\mathfrak {n}}}

\newcommand{\liep}{{\mathfrak {p}}}
\newcommand{\lieq}{{\mathfrak {q}}}

\newcommand{\liet}{{\mathfrak {t}}}
\newcommand{\lieu}{{\mathfrak {u}}}

\newcommand{\liegl}{{\mathfrak {gl}}}
\newcommand{\liesl}{{\mathfrak {sl}}}
\newcommand{\lieso}{{\mathfrak {so}}}

\theoremstyle{plain}
\newtheorem{theorem}{Theorem}[section]
\newtheorem{lemma}[theorem]{Lemma}
\newtheorem{corollary}[theorem]{Corollary}
\newtheorem{proposition}[theorem]{Proposition}
\newtheorem{conjecture}[theorem]{Conjecture}

\theoremstyle{remark}

\bibpunct{[}{]}{,}{n}{}{,}

\begin{document}

{\small
\title{Algebraic Characters of Harish-Chandra Modules}
\author{Fabian Januszewski}
\maketitle
}
\begin{abstract}
We give a cohomological treatment of a character theory for $(\lieg,K)$-modules. This leads to a nice formalism extending to large categories of not necessarily admissible $(\lieg,K)$-modules. Due to results of Hecht, Schmid and Vogan the classical results of Harish-Chandra's global character theory extend to this general setting. As an application we consider a general setup, for which we show that algebraic characters answer discretely decomposable branching problems.
\end{abstract}

\section*{Introduction}

In a series of fundamental papers \cite{harishchandra1953,harishchandra1954a,harishchandra1954b} Harish-Chandra initiated the theory of $(\lieg,K)$-modules and proved the existence and fundamental properties of distribution characters for admissible $(\lieg,K)$-modules under the assumption of boundedness condition for the multiplicities of $K$-types. Harish-Chandra's global characters are a central tool in the study of Harish-Chandra modules.

In this paper we discuss a cohomological algebraic definition of the notion of character essentially for arbitrary $(\lieg,K)$-modules. In fact even the admissibility condition can be dropped such that algebraic characters extend to larger categories than analytic global characters do. Of course this introduces complications, and apart from showing the fundamental properties of algebraic characters, our goal is to trace out the merits and limitations of this approach.

That such an algebraic theory exists might not be surprising to the experts, as it is essentially an algebraic formalization of Harish-Chandra's fundamental work.

The connection between characters and cohomology seems to have been observed first by Bott \cite{bott1957} in the finite-dimensional case, and was later refined by Kostant \cite{kostant1961}, who interpreted the Weyl character formula in terms of Euler characteristics. In a broader context the connection between Harish-Chandra's global characters and cohomology had been conjectured first by Osborne in his thesis \cite{osborne1972}, later resolved by Hecht and Schmid \cite{hechtschmid1983} in the real and Vogan \cite{vogan1979ii} in the $\theta$-stable cases respectively.

We exploit that, thanks to the work of many people, we can rely on a fairly complete picture concerning the fundamental properties of $\lieu$-cohomology and its interplay with global characters. This fills this seemingly bloodless abstract theory with life (cf.\ Theorem \ref{thm:main2} below).

Our construction proceeds as follows. For a germane parabolic subalgebra $\lieq\subseteq\lieg$ (as by Knapp and Voganin \cite{book_knappvogan1995}) with nilpotent radical $\lieu$ we introduce the notion of $\lieu$-admissible pair of categories $(\mathcal G,\mathcal L)$. Here $\mathcal G$ is a category of $(\lieg,K)$-modules and $\mathcal L$ is a category of $(\liel,L\cap K)$-modules for the $\theta$-stable Levi factor $\liel$ of $\lieq$. Then $\lieu$-admissibility guarantees, apart from another technical assumption, that the $\lieu$-cohomology of objects in $\mathcal G$ lies in $\mathcal L$ and enables us to define the $\mathcal L$-valued characters of objects in $\mathcal G$ essentially by the Euler characteristic of $\lieu$-cohomology, divided by a canonical Weyl denominator.

Our characters live in a localized Grothendieck group $K(\mathcal L)$ of $\mathcal L$. For representations in $\mathcal G$ that are already in $\mathcal L$ it turns out, a posteriori, that the characters essentially lie in the unlocalized Grothendieck group, in the sense that the representations themselves give a canonical preimage in the unlocalized Grothendieck group, which maps to the character in the localization. In this sense cohomological characters generalize the naive algebraic notion of character.

We remark that what we term {\em localization} in our context is to be distinguished from Beilinson-Bernstein localization. The parabolic $\lieq$ uniquely determines a {\em Weyl denominator} $W_\lieq$ in $K(\mathcal L)$ and under our assumptions $K(\mathcal L)$ becomes a $\ZZ[W_\lieq]$-module. Then formally our characters live in
$$
K(\mathcal L)[W_\lieq^{-1}]\;=\;\ZZ[W_\lieq,W_\lieq^{-1}]\otimes_{\ZZ[W_\lieq]}K(\mathcal L).
$$

A practically useful property of $\lieu$-cohomology is that it is infinitely additive and preserves $Z(\lieg)$- resp.\ $Z(\liel)$-finiteness. However in general it does not preserve admissibility. In order to remedy for this we introduce the notion of {\em contructible parabolic subalgebras} $\lieq\subseteq\lieg$ and show that in this case $\lieu$-cohomology preserves finite length. Furthermore the conjugates of the Levi factors $L\subseteq G$ of the contructible Borel algebras cover the entire group. This enables us to carry over the {\em linear independence} of Harish-Chandra's global characters to our setting.

The notion of constructible parabolic provides is with several classes of $\lieu$-admissible modules. For a constructible parabolic $\lieq$ the the categories of finite length modules as well as discretely decomposable modules with suitable multiplicity constraints are $\lieu$-admissible.

The cohomological notion of character leads to a nice formalism, which follows from purely cohomological arguments (cf.\ Theorem \ref{thm:main1} below). In particular it is additive, multiplicative, respects duals, is transitive. By transitivity of characters we mean the fact that the character of a character is the character we'd expect. Our characters extends both the naive notion of algebraic character as well as Harish-Chandra's, and behaves well in coherent families, i.e.\ under translation functors. Furthermore it gives an approach to not necessarily admissible branching problems (cf.\ Proposition \ref{prop:restriction} and Theorem \ref{thm:main5} below).

For example it becomes clear in the algebraic picture that Blattner formulae are consequences of character formulae and in the case of the discrete series even equivalent, provided all irreducible constituents are sufficiently regular, cf.\ Theorem \ref{thm:main4} below. However we conjecture that our regularity condition is satisfied by all discrete series representations, and more generally by all representations with non-trivial $(\lieg,K)$-cohomology. We will come back to this question in the future. An interesting novelty is that similar statements hold for more general branching problems.

For compact groups it is a well known fact that we recover the classical algebraic notion of character by considering a Borel subalgebra. Along these lines, as is well known, the Weyl character formula in the classical picture is essentially equivalent to (a special case of) Kostant's Theorem on the structure of $\lieu$-cohomology. The transitivity of algebraic characters is reflected in the general statement of Kostant's Theorem for not necessarily minimal parabolic subalgebras. As a consequence, from a formal point of view, in our theory the case of a maximal parabolic is already enough to prove the Weyl character formula.

The same statement remains true in general, cf.\ Proposition \ref{prop:ccomposition}. By the same general principle of transitivity
fundamental results of Hecht, Schmid and Vogan \cite{hechtschmid1983,vogan1979ii} imply that cohomological characters formally coincide with Harish-Chandra's for modules of finite length and hence characterize multiplicities of composition factors uniquely.

The characterization of composition factors by characters remains true for larger classes of representations, but they become less sharp, as localization at the Weyl denominator on the level of Grothendieck groups is not a faithful operation. The kernel of the natural map
$$
K(\mathcal L)\;\to\;
K(\mathcal L)[W_\lieq^{-1}]
$$
consists of all elements of $K(\mathcal L)$ annihilated by a power of $W_\lieq$.

In the world of finite length modules localization does no harm, i.e.\ the product of a collection of characters associated to constructible parabolics is always injective on the Grothendieck group whenever their Levi factors cover the entire group. This situation becomes more involved once we allow discretely decomposable modules, which is a natural setting when approaching branching problems from an algebraic point of view.

As an elementary example we consider the category ${\mathcal C}_{\rm fl}$ of finite length $(\liesl_2,\SO(2))$-modules, and choose $\lieq=\liel+\lieu$ a minimal $\theta$-stable parabolic subalgebra with a Levi decomposition as indicated. Then our characters are multiplicative in the sense that the product of a finite length module $M$ with a finite-dimensional module $F$ has character
$$
c_{\lieq}(M\otimes_\CC F)=c_\lieq(M)\cdot c_\lieq(F).
$$
This is an identity in the Grothendieck group of finite length modules of $\liel$, localized at
$$
W_\lieq:=1-[\alpha],
$$
where $-\alpha$ is the weight of $\liel$ occuring in $\lieu$, and $[-\alpha]$ being its class. In particular this may be interpreted an identity in a rational function field $\QQ(T)$ after identifying $[\alpha]$ with the transcendental variable $T$, and as the map $\QQ[T]\to\QQ(T)$ is injective, we lose no information there, and this is true more generally for any reductive pair.

Now if $N$ is another finite length module, we may be interested in the character of $M\otimes_\CC N$. However, this tensor product is no object in ${\mathcal C}_{\rm fl}$, and in general it is even not in the category of discretely decomposables ${\mathcal C}_{\rm df}$ with finite multiplicities. However if $M$ and $N$ are both discrete series representations with the property that their $\SO(2)$-types lie in the same $\liesl_2$-Weyl chamber, say are of weight $\frac{(2+n)}{2}\cdot \alpha$, and $n\geq 0$, then $M\otimes_\CC N$ lies in ${\mathcal C}_{\rm df}$. Our formalism therefore says that the identity
$$
c_\lieq(M\otimes_\CC N)=c_\lieq(M)\cdot c_\lieq(N)
$$
is true, this time formally in $\ZZ[[\sqrt{T},\sqrt{T}^{-1}]][\frac{1}{1-T}]$, the module of unbounded Laurent series, localized at the Weyl denominator. Here the natural map
$$
\ZZ[[\sqrt{T},\sqrt{T}^{-1}]]\;\;\to\;\; \ZZ[[\sqrt{T},\sqrt{T}^{-1}]][\frac{1}{1-T}]
$$
is no more injective, even on the subgroup of Weyl numerators. For example the $\lieq$-character of $[D_1]- [D_{-1}]$ lies in the kernel, where $D_{\pm1}$ denotes the limits of discrete series representation with lowest $\SO(2)$-type of weight $\pm\frac{\alpha}{1}$. Therefore at this stage we are unable to determine the $\lieq$-character (i.e.\ the composition factors with non-trivial $\lieq$-characters) of $M\otimes_\CC N$ uniquely.

Nonetheless there is a way out. Consider the full subcategory ${\mathcal C}_{\rm df}^+$ of ${\mathcal C}_{\rm df}$ of modules subject to the same $\SO(2)$-type condition as $M$ and $N$. Then it is even true that $M\otimes_\CC N$ is an object in ${\mathcal C}_{\rm df}^+$, which is easily seen by solving the branching problem for $\SO(2)$, and even though the above localization map is not injective, it is easy to see in this example that the map $c_\lieq$ is injective when considered as a map from the Grothendieck group of ${\mathcal C}_{\rm df}^+$ to the above localization. Therefore we may indeed solve the above branching problem:
$$
c_\lieq(D_m\otimes_\CC D_n)=
c_\lieq(D_m)\cdot c_\lieq(D_n)=
\frac{T^{\frac{m}{2}}}{1-T}\cdot
\frac{T^{\frac{n}{2}}}{1-T}=
$$
$$
\sum_{k=0}^\infty
\frac{T^\frac{m+n+2k}{2}}{1-T}=
\sum_{k=0}^\infty
c_\lieq(D_{m+n+2k}),
$$
with the notation $D_m$ for the discrete series of lowest $\SO(2)$-type $m\geq 2$ as above.

In summary we reduced this branching problem to the following two statements:
\begin{itemize}
\item[(C)] {\em Containedness:} The restricted module $M\otimes_\CC N$ lies in ${\mathcal C}_{\rm df}^+$.
\item[(I)] {\em Injectivity:} The kernel of $c_\lieq$ is trivial on the Grothendieck group of ${\mathcal C}_{\rm df}^+$.
\end{itemize}

In the second part of this paper we present a general approach to the construction of a large category ${\mathcal C}^+$ satisfying (I), starting from a category $\mathcal C\subseteq\mathcal C_{\rm df}$, which at the same time gives a criterion for checking (C) for objects in $\mathcal C$.

However in general the formulation of (C) and (I) is more involved, as one single $\lieq$ is no more sufficient, and in this sense (C) and (I) should be understood as statements about a collection of characters for (all) different classes of parabolics.

As a complementary case, consider for example the case where $\pi$ is a principal series representations of $(\liesl_2,\SO(2))$, and we are interested in its restriction to $\SO(2)$. Then we still have the identity
$$
c_\lieq(\pi)=
\iota(\pi),
$$
in the appropriate localization, where $\iota$ denotes the restriction to $\SO(2)$. Then we know that the left hand side vanishes, and so does the right hand side --- in the localization. Therefore we know that the right hand side lies {\em in} the kernel of the localization map. However this kernel may be explicitly computed (which is a simple exercise here), and the following two assertions allow us to determine the decomposition:
\begin{itemize}
\item[(B)] {\em Boundedness:} Multiplicities of the $\SO(2)$-types occuring in $\pi$ is bounded by a constant.
\item[(S)] {\em Sample:} The multiplicities of $0\cdot\alpha$ and $\frac{1}{2}\cdot\alpha$ in $\pi$ are known.
\end{itemize}
Then we may conclude that
$$
\iota(\pi)=\sum_{k\in\ZZ}\left[\frac{(2k+\delta)\cdot\alpha}{2}\right]
$$
where $\frac{\delta}{2}\cdot\alpha$ is the weight occuring in $\pi$.

Note that $0\cdot\alpha$ and $\frac{1}{2}\cdot\alpha$ both do {\em not} occur in a discretely decomposable $(\liesl_2,\SO(2))$-module $M$ if and only if all its composition factors belong to the discrete series. This is the criterion we have in mind to check (C) above. It is easy to see that for this enlarged category (I) still holds, and it is a maximal subcategory of $\mathcal \pi_{\rm df}$ satisfying this property.

The philosophy behind this example is that the $\lieq$-character in the localization plus the additional information is precisely what we need to solve our branching problem (here a Blattner formula) for a general input.

In nature such an instance is given for example by Schmid's upper bound on the $K$-types for the discrete series \cite[Theorem 1.3]{schmid1975}. His result, plus the character formula on a fundamental Cartan yield the Blattner conjecture for the discrete series, subject to a regularity condition (cf.\ condition (S) in section \ref{sec:blattner}), that we conjecture to be always satisfied, thus possibly giving a new proof of the Blattner formula.

In the general picture we are naturally led to study the kernels of the localization maps, and so far we are only able to treat the absolute case (i.e.\ $\lieq$ minimal), which nonetheless is the most important one from the classical perspective. We show that vanishing in the localization forces certain simple symmetries in the character, and those yield the existence of certain irreducible constituents which forms the sample set for (S) above.

In general condition (B) should read {\em bounded by a polynomial of fixed degree in the norm of the infinitesimal character} (for a precise statement, see condition \eqref{eq:multiplicitybound} in section 5), and the sample in (S) depends on this degree, and is usually not finite (cf.\ Theorem \ref{thm:main5} and Corollary \ref{cor:kernelmultibasis}).

A prototypical example for (B) is Harish-Chandra's bound for the multiplicities of $K$-types in finite length representations --- they are bounded by their dimensions, the latter in turn being explicitly computable via the Weyl dimension formula departing from the infinitesimal character.

The localization problem is analogous to the well known classical situation, where the restriction of the global character may vanish on the regular elements, thus making it difficult to extract the desired information.


From the algebraic perspective a relative treatment would be desirable, as this would circumvent the vanishing problem, in the same way as David Vogan's approach to minimal $K$-types \cite{vogan1979} avoids this by considering non-minimal parabolic subalgebras whenever necessary. Our approach is similar yet different, as we rely on the same spectral sequence, but do not focus on a particular $K$-type and also replace $K$ by any reductive pair.

In the context of a general branching problem, the above observations lead us to the following slightly more effective version of Kobayashi's Conjecture C in \cite{kobayashi2000}.
\begin{conjecture}\label{conj:polybound}
Let $(G,G')$ be a semisimple symmetric pair, and $\pi\in\hat{G}$ an irreducible unitary representation of $G$. Assume that the restriction of $\pi$ to $G'$ is infinitesimally discretely decomposable, then the dimension
$$
\dim\Hom_{G'}(\tau,\pi|_{G'}),\;\;\;\tau\in\hat{G}'
$$
is finite and grows at most polynomially in the norm of the infinitesimal character of $\tau$.
\end{conjecture}
As the character formula for the Zuckerman-Vogan cohomological induction modules $A_\lieq(\lambda)$ is known, we are optimistic that our approach may be applied to produce more evidence towards Kobayashi's Multiplicity-Free Conjecture
and also its relation to the virtually symmetric type, cf.\ Conjectures 4.2 and 4.3 in \cite{kobayashi2011}, at least in the infinitesimally discretely decomposable cases. We hope to come back to this in the future.

Our study was initially motivated by questions of non-vanishing of periods attached to automorphic representations with applications to number theory. This problem turns out to be morally equivalent to suitable multiplicity-one statements for non-admissible restrictions of $(\lieg,K)$-modules. As the flavor of this problem is more algebraic than analytic, we hope that the theory proposed here may serve as a first step toward a general approach to this type of questions.

Our theory generalizes to other contexts as well. In particular we may consider Michael Harris' notion of Beilinson-Bernstein localization over $\QQ$, which essentially studies $(\lieg,K)$-modules with an additional action of the absolute Galois group, and has applications to periods of automorphic forms as well \cite{harris_preprint}. Our theory easily generalizes to this setup, yielding characters with Galois actions.

The paper is organized as follows. In a rather long zeroth section we collect well known facts about certain categories of $(\lieg,K)$-modules and their cohomology and extend them to discretely decomposable modules whenever possible. In the first section we introduce the abstract notion of cohomological characters. In the second section we translate known results about the $\lieu$-cohomology into this setting and show that translation functors essentially commute with characters, also allowing appropriate categories of discretely decomposable modules. In the third section we give applications of the theory and show that algebraic characters determine composition factors of finite length modules uniquely. In section four we treat the problem of reading off Blattner formulae from character formulae, and in section five we generalize our results to discretely decomposable modules with polynomial multiplicity bounds. In the sixth section we examine the kernels of the localization maps, which is relevant for the study of Blattner formulae and more generally for discretely decomposable branching problems.

{\bf Acknowledgments.} The author thanks Roger E.\ Howe for helpful discussions and he thanks David Vogan for providing the reference to his proof of the analogue of Osborne's conjecture in the $\theta$-stable case. The author thanks the refree for valuable comments and corrections. The author acknowledges support from the German Academy of Sciences Leopoldina grant no.\ LDPS 2009-23, and he also thanks the mathematics department at University of California at Los Angeles, where a substantial part of this work was done.

\section*{Notation and terminology}

The concise monograph \cite{book_knappvogan1995} contains most of the basic notions and results we need.

\subsection*{Reductive pairs}

Throughout the paper we fix a reductive pair $(\lieg,K)$ where $\lieg$ is the complexified Lie algebra of a real Lie algebra $\lieg_0$ and $K$ is a maximally compact subgroup in a reductive Lie group $G$ with Lie algebra $\lieg_0$. The group $G$ then has finitely many components and we denote $G^0$ the connected component of the identity in $G$. We write $\liek\subseteq\lieg$ for the complexification of the Lie algebra $\liek_0\subseteq\lieg_0$ of $K$, $U(\lieg)$ for the universal enveloping algebra of $\lieg$, and $Z(\lieg)$ for its center.

There is a natural dictionary between the theory of such reductive groups $G$ and reductive pairs $(\lieg,K)$, i.e.\ for a reductive pair there is a unique $G$ and vice versa, preserving finite-dimensional representations, cf.\ \cite[Chap.\ IV]{book_knappvogan1995}. The group $G$ comes with a Cartan involution $\theta$, which is stricly speaking also part of the datum $(\lieg,K)$, as is the invariant bilinear form coming from $G$ and the real Lie algebra $\lieg_0$.

We assume that our parabolic subalgebras $\lieq$ of $\lieg$ are always germane in the sense of \cite[Chap. IV]{book_knappvogan1995}, i.e.\ they possess a Levi factor $\liel$ which is the complexification of a $\theta$-stable real $\liel_0\subseteq\lieg_0$. Then $\liel_0$ is the Lie algebra of the closed reductive subgroup $L\subseteq G$ given by the intersection of the normalizers of $\lieq$ and $\theta(\lieq)$ in $G$ and $L$ normalizes $\lieu$. In this context Levi factors and Cartan subalgebras are always assumed $\theta$-stable and defined over $\RR$. The same terminology applies to parabolic and Cartan pairs. In particular for a Cartan subpair $(\lieh,T)$ we have a corresponding subgroup $H\subseteq G$ with $T=H\cap K$. The set of roots of $\lieh$ in $\lieg$ is denote by $\Delta(\lieg,\lieh)$, and $\rho(\lieu)$ denotes the half sum of the weights of $\lieh$ in $\lieu$ ($\lieh$ is always clear from the context).

If $\liel$ is an abelian Lie algebra and $\lambda\in\liel^*$ is a character, we write $\CC_\lambda$ for the one-dimensional representation space of $\lambda$. The trivial representation is denoted simply by $\CC$. The reader familiar with the classical picture hopefully accepts our apologies for our consequent ignorance of the classical analytic notation.

In many situations we concentrate on the case of connected $K$. However most statements carry over to the non-connected case. In this context the reader may consult Chap.\ IV Sec.\ 2 in \cite{book_knappvogan1995} for an account of Cartan-Weyl's highest weight theory for non-connected groups and Sec.\ 8 of loc.\ cit.\ for infinitesimal characters. However in applications the component groups may pose non-trivial problems that need to be dealt with.

\subsection*{$(\lieg,K)$-modules}

If $X$ is an irreducible $(\lieg,K)$-module, then it is admissible \cite{lepowsky1973}. A theorem of Dixmier says that $X$ then has an infinitesimal character. Hence if $X$ is a $(\lieg,K)$-module of finite length, it is necessarily admissible and $Z(\lieg)$-finite. If $X$ is a $\lieg$-module with an action of $K$, we write $X_K$ for the subspace of $K$-finite vectors. Then the functor $(\cdot)_K$ is left exact, but not exact in general.

In the literature the term Harish-Chandra module comes in several variations. We understand it synonymously to $(\lieg,K)$-module. The latter notion imposes a priori no finiteness condition except the mandatory local $K$-finiteness. We denote by ${\mathcal C}(\lieg,K)$ resp.\ ${\mathcal C}_{\rm a}(\lieg,K)$ resp.\ ${\mathcal C}_{\rm fd}(\lieg,K)$ resp.\ ${\mathcal C}_{\rm fl}(\lieg,K)$ the categories of all resp.\ admissible resp.\ finite-dimensional resp.\ finitely generated admissible $(\lieg,K)$-modules. Note that the latter category coincides with the categories of modules of finite length resp.\ the category of admissible $Z(\lieg)$-finite modules.

We say that a $(\lieg,K)$-module $X$ has an irreducible $(\lieg,K)$-module $Y$ as a {\em composition factor} if there are submodules $X_1\subseteq X_0\subseteq X$ such that $X_1/X_0\cong Y$. We write $S(X)$ for the set of submodules of $X$ and consider it as a preordered set via the subset relation. The {\em multiplicity} of $Y$ in $X$ is the supremum $m_Y(X)$ of the cardinalities of all totally ordered sets $(I,\leq)$ with the property that there exist injective order-preserving maps $a:I\to S(X)$ and $b:I\to S(X)$ such that for any $i\in I$ we have $a(i)\subseteq b(i)$ and $b(i)/a(i)\cong Y$.

We write ${\mathcal C}_{\rm d}(\lieg,K)$ for the category of discretely decomposable modules as introduced by Kobayashi \cite[Definition 1.1]{kobayashi1997}, i.e.\ modules that are direct limits of finite length modules. This category is a full abelian subcategory of ${\mathcal C}(\lieg,K)$. We introduce another category ${\mathcal C}_{\rm f}(\lieg,K)$ as the full subcategory of $(\lieg,K)$-modules $X$ with the property that the multiplicity $m_Y(X)$ of any irreducible $Y$ in $X$ is finite. Then this category is again abelian and we denote by ${\mathcal C}_{\rm df}(\lieg,K)$ the intersection of the latter category with the category of discretely decomposable modules.

We write $K_{?}(\lieg,K)$, $?\in\{{\rm f,a,fl,df}\}$ for the corresponding Grothendieck group of ${\mathcal C}_{?}(\lieg, K)$. We write $[X]$ for the class of $X$ in $K_{?}(\lieg,K)$. It is crucial that the addition law in $K_{?}(\lieg,K)$ is only {\em finite}, i.e.\ comes from splitting short exact sequences. In ${\mathcal C}_{?}(\lieg,K)$, $?\in\{-,{\rm fl}\}$ we have a duality sending $X$ to its locally $K$-finite dual $X^*$. Being exact this duality naturally extends to $K_{?}(\lieg,K)$ where $?\in\{-,{\rm fl}\}$ (the Grothendieck group being trivial for $?=-$).

By the above, for fixed $Y$ the numbers $m_X(Y)$ have a natural continuation to $K_{?}(\lieg,K)$ that is additive in the variable $X$ and satisfies $m_Y([X])=m_Y(X)$. As any non-zero $X$ in $\mathcal C_?(\lieg,K)$, $?\in\{\rm a,fl,df\}$ has a non-zero composition factor with finite multiplicity, we see that in the respective groups $[X]$ is zero if and only if $X$ is zero. Furthermore $[X]=[Y]$ implies that $X$ and $Y$ have the same composition factors. In that sense the Grothendieck group $K_{?}(\lieg,K)$ is well behaved. We remark that as restriction along a map of reductive pairs $(\liel,L\cap K)\to(\lieg,K)$ is exact it descends to the Grothendieck groups.

As ${\mathcal C}_{?}(\lieg,K)$ is not closed under tensor products, we only have a partially defined commutative multiplication which is (finitely) distributive in the obvious way. Localization is well behaved in the following sense. Writing $\CC$ for the trivial representation, we get $[\CC]=1$, hence a multiplicative unit exists in $K_?(\lieg,K)$. If an element $0\neq D\in K_?(\lieg,K)$ can be multiplied with any element $C\in K_?(\lieg,K)$ then the localization $K_?(\lieg,K)[D^{-1}]$ is well defined. Generally we can localize at any non-zero linear combination of one-dimensional representations. In the case $K_{\rm fl}(\lieg,K)$ we can localize at any non-zero linear combination of finite-dimensional representations. Due to a result of Kostant \cite{kostant1975}, \cite[Theorem 7.133]{book_knappvogan1995}, this is also true in $\mathcal C_{\rm a,\rm df}(\lieg,K)$. As we may have zero divisors the canonical map $K_?(\lieg,K)\to K_?(\lieg, K)[D^{-1}]$ is usually far from being injective.

We need the following
\begin{lemma}\label{lemma:grothendieckcommute}
Let $P:\mathcal C(\lieg,K)\to\mathcal C(\lieg',K')$ be left adjoint to an exact covariant functor $F:\mathcal C(\lieg',K')\to\mathcal C(\lieg,K)$. Then the left derived functors of $P$ commute with direct limits.
\end{lemma}

\begin{proof}
As $P$ is a left adjoint, it commutes with direct limits. Hence we have for the $q$-th left derived functors
\begin{equation}
L^q(\varinjlim P)(X_i)=L^q(P\varinjlim)(X_i).
\label{eq:limcomm}
\end{equation}
It is enough to show the existence of two Grothendieck spectral sequences, one converging to the left hand side, one to the right hand side. Those spectral sequences will collapse by the exactness of direct limits. For the left hand side, nothing is to show as any object is acyclic for the direct limit. For the right hand side we can choose for each object $X_i$ a resolution of standard projectives in the sense of \cite[Section II.2]{book_knappvogan1995}. As the construction of these projectives proceeds by production, which commutes with direct limits, we see that the direct limit of a standard projective is (again a standard) projective. Hence we have a Grothendieck spectral sequence
$$
(L^{-p}I)L^{-q}\varinjlim X_i\;\Longrightarrow\;L^{-p-q}(I\varinjlim)(X_i).
$$
The edge morphisms of the two spectral sequences yield isomorphisms
$$
(L^{-q}P)\varinjlim X_i\;\cong\;L^{-q}(P\varinjlim)(X_i)\;\cong\;
L^{-q}(\varinjlim P)(X_i)\;\cong\;\varinjlim L^{-q}P(X_i).
$$
This proves the claim.
\end{proof}

\begin{corollary}\label{cor:discretecohomology}
Taking homology as considered below commutes with direct limits. More generally the $\Ext$-functors commute with direct limits in the first argument.
\end{corollary}

If $K$ is not connected, we consider the natural action of $K$ on the set of $\mathcal Z_\lieg$ homomorphisms $Z(\lieg)\to\CC$. Then a {\em (generalized) infinitesimal character} is a $K$-orbit in $\mathcal Z_\lieg$.

\begin{proposition}[Wigner's Lemma]
Let $\chi$ be an infinitesimal character of a $(\lieg,K)$-module $X$ and let $Y$ be a discretely decomposable $(\lieg,K)$-module whose composition factors have infinitesimal characters $\neq \chi$. Then
\begin{equation}
{\rm Ext}^n_{\lieg,K}(Y,X)=0
\label{eq:wigner}
\end{equation}
for all $n$.
\end{proposition}

\begin{proof}
This is a consequence of the classical Wigner Lemma (Proposition 7.212 in loc.\ cit.) which says in our setting that if we write $Y=\varinjlim Y_i$ with $Y_i$ of finite length then ${\rm Ext}^n_{\lieg,K}(Y_i,X)=0$. As the functor ${\rm Ext}^n_{\lieg,K}(X,\cdot)$ commutes with injective limits by Lemma \ref{lemma:grothendieckcommute}, the identity \eqref{eq:wigner} follows.
\end{proof}

As a consequence of Yoneda's description of ${\rm Ext}^n_{\lieg,K}(Y,X)$ as the group of classes of $n$-extensions of $X$ and $Y$ Wigner's Lemma tells us that any discretely decomposable module $Y$ decomposes into the direct sum of its $\chi$-primary components, where $\chi$ runs through the infinitesimal characters. In particular $\chi$-primary components are well defined for discretely decomposable modules.

We have an explicit description of the projection $p_\chi$ onto the $\chi$-primary component. Any element $y\in Y$ has a preimage $y_i$ in some $Y_i$. We can consider the projection $p_i$ of $Y_i$ to its $\chi$-primary component (cf.\ Proposition 7.20 in loc.\ cit.). Then the elements $p_i(y_i)$ can be assumed to be compatible elements of the directed system of the $p_i(Y_i)$ and hence their limit is well defined as an element of the injective limit of the $\chi$-primary components of the $Y_i$.

\begin{corollary}\label{cor0:wigner}
For any discretely decomposable $(\lieg,K)$-module $X$ there is a canonical decomposition
\begin{equation}
X\cong\bigoplus\limits_{\chi} X_\chi
\label{eq:xiprimary}
\end{equation}
where $\chi$ ranges over the infinitesimal characters of composition factors of $X$ and $X_\chi$ is the $\chi$-primary component. Furthermore $X_\chi$ is the direct limit of the $\chi$-primary components $(X_i)_\chi$ of the $X_i$ and $X_\chi$ is of finite length if $X$ lies in $\mathcal C_{\rm df}(\lieg,K)$.
\end{corollary}

\begin{proof}
The decomposition \eqref{eq:xiprimary} follows for each finite level from Wigner's Lemma. As direct sums and direct limits commute, \eqref{eq:xiprimary} holds as stated. The statement about finite length follows from Harish-Chandra's celebrated theorem that, up to isomorphy, there are only finitely many irreducible $(\lieg,K)$-modules sharing the same infinitesimal character.
\end{proof}

\subsection*{Lie algebra cohomology}

For the convenience of the reader we recall known facts about Lie algebra cohomology. Let $\lieg$ be a Lie algebra, $\lieu\subseteq\lieg$ a subalgebra and $V$ a $\lieg$-module, naturally considered as a $\lieu$-module as well.

The cohomology
$$
H^\bullet(\lieu;V)
$$
may be calculated from the finite-dimensional standard complex
$$
\Hom_\CC(\bigwedge^\bullet\lieu;V)
$$
whose differential is known explicitly. Dually the homology
$$
H_\bullet(\lieu;V)
$$
may be calculated from the finite-dimensional standard complex
$$
(\bigwedge^\bullet\lieu)\otimes_\CC V,
$$
again with explicit differential. We remark that by hard duality \cite[Corollary 3.8]{book_knappvogan1995} we have natural isomorphisms
\begin{equation}
H_q(\lieu;V\otimes_\CC(\bigwedge^{\rm top}\lieu)^*)\cong
H^{{\rm top}-q}(\lieu;V).
\label{eq:hardduality}
\end{equation}
In particular homology {\em and} cohomology both commute with direct limits.

Now assume that $\lieg$ is complex reductive, and that $\lieq\subseteq\lieg$ is a parabolic subalgebra with Levi decomposition $\lieq=\liel+\lieu$, $\liel$ being a Levi factor and $\lieu$ the nilpotent radical. Then we have a natural action of $\liel$ on the above standard complex. The differential $d$ of this complex turns out to be $\liel$-linear.

Let $\lieh\subseteq\liel$ be a Cartan subalgebra. We identify characters of $Z(\liel)$ via the Harish-Chandra map with characters of $U(\lieh)^{W(\liel,\lieh)}$.
\begin{proposition}[Theorem 7.56 of loc.\ cit.]\label{prop0:finiteness}
Let $X$ be a discretely decomposable (resp.\ $Z(\lieg)$-finite) $\lieg$-module. Then $H^q(\lieu;V)$ is a discretely decomposable (resp.\ $Z(\liel)$-finite) $\liel$-module and if its $\chi$-primary component for the action of $Z(\liel)$ is non-zero then
$$
\chi=\chi_\nu
$$
where $\nu=w\lambda-\frac{1}{2}\Delta(\lien,\lieh)$ with some $w\in W(\lieg,\lieh)$ and the $\chi_\lambda$-primary component of $V$ is non-zero.
\end{proposition}

\begin{proof}
As modules of finite length are $Z(\lieg)$-finite, this easily reduces to the case of a $Z(\lieg)$-finite module $V$, which in turn is discussed in Theorem 7.56 in \cite{book_knappvogan1995}.
\end{proof}

Let us now return to the general setting, i.e.\ $(\lieg,K)$-modules $V$, where $(\lieg,K)$ is a reductive pair. Similarly $(\liel,L\cap K)$ is the reductive pair associated to the Levi factor of a germane parabolic subalgebra $\lieq$. We have natural actions of $(\liel,L\cap K)$ on the above standard complex, which descends to cohomology. This action is natural in the following sense. Denote $p:(\lieq,L\cap K)\to (\liel,L\cap K)$ the canonical projection. Then $p$ induces an exact forgetful functor $\mathcal F(p):{\mathcal C}(\liel,L\cap K)\to{\mathcal C}(\lieq,L\cap K)$. It turns out that this functor has a right adjoint $I(p)$, whose composition with the forgetful functor along $(\lieq,L\cap K)\to(\lieg,K)$ furnishes a left exact functor $H^0:{\mathcal C}(\lieg,K)\to{\mathcal C}(\liel,L\cap K)$ given by the subspace of $\lieu$-invariants. The right derived functors of $H^0$ are naturally isomorphic to $H^q(\lieu;\cdot)$ as universal $\delta$-functors.
\begin{proposition}[Corollary 5.140 of loc.\ cit.]\label{prop0:admissibility}
If $X$ is an admissible $(\lieg,K)$-module and $\lieq$ is $\theta$-stable, then $H^q(\lieu;X)$ is an admissible $(\liel,L\cap K)$-module.
\end{proposition}

\begin{corollary}\label{cor0:finitelength}
If $X$ is a finite length $(\lieg,K)$-module and $\lieq$ is $\theta$-stable, then $H^q(\lieu;X)$ is of finite length as $(\liel,L\cap K)$-module.
\end{corollary}

\begin{proof}
An admissible module is of finite length if and only if it is $Z(\lieg)$-finite. Proposition \ref{prop0:finiteness} concludes the proof.
\end{proof}

\begin{proposition}[Hecht-Schmid \cite{hechtschmid1983}]\label{prop0:finitelength}
If $X$ is a finite length $(\lieg,K)$-module and $\lieq$ is real, then $H^q(\lieu;X)$ is of finite length as $(\liel,L\cap K)$-module.
\end{proposition}

\begin{proposition}[K\"unneth Formula]
Let $V$ and $W$ be two $(\lieg,K)$-modules, then
$$
\bigoplus_{p+q=n}H^p(\lieu;V)\otimes H^q(\lieu;W)\;\cong\;
H^{n}(\lieu\times\lieu;V\otimes_\CC W),
$$
in the category of $(\liel\times\liel,L\cap K\times L\cap K)$-modules.
\end{proposition}

\begin{proof}
The proof is standard.
\end{proof}

\begin{proposition}[Poincar\'e Duality]\label{prop0:poincare}
For any $(\lieg,K)$-module $V$ we have a natural isomorphism
$$
H^q(\lieu;V)^*\;\cong\;H^{\dim\lieu-q}(\lieu;V^*\otimes\bigwedge^{\dim\lieu}\lieu),
$$
where the dual on the right hand side is understood in the category of $(\lieq,L\cap K)$-modules.
\end{proposition}

\begin{proof}
The proof is standard.
\end{proof}

\begin{theorem}[Hochschild-Serre \cite{hochschildserre1953}]\label{thm0:hs}
Let $\lieg$ be an arbitrary complex Lie algebra, $\lieh\subseteq\lieg$ be a Lie subalgebra, $V$ a $\lieg$-module that we also consider as a $\lieh$-module. We have a cohomological spectral sequence with
$$
E_1^{p,q}=\bigwedge^p(\lieg/\lieh)^*\otimes H^q(\lieh,V),
$$
and
$$
E_1^{p,q}\Longrightarrow H^{p+q}(\lieg, V).
$$
\end{theorem}
In all situations that we are interested in, this spectral sequence eventually respects the additional module structures on cohomology, similarly to the K\"unneth formula.

For compact connected $G$ and $H=T$ a maximal torus we have the fundamental
\begin{theorem}[{Special case of Kostant's Theorem \cite{kostant1961}, \cite[Theorem 4.135]{book_knappvogan1995}}]
Let $G$ be compact connected, $V$ be an irreducible finite-dimensional representation of $G$ of highest weight $\lambda$ with respect to $T$. Then the $T$-module $H^q(\lieu;V)$ decomposes into one-dimensional spaces with weights
$$
w(\lambda+\rho(\lieu))-\rho(\lieu),
$$
with $w\in W(\lieg,\liet)$ of length $\ell(w)=q$, each occuring with multiplicity one.
\end{theorem}

This fundamental theorem has a natural generalization to arbitrary parapolic $\lieu\subseteq\lieg$, also given by Kostant, which fits well into our relative character theory below.

\section{Algebraic characters}

In this section $(\lieg,K)$ denotes a reductive pair, $(\lieq,L\cap K)$ is a germane parabolic subpair with Levi factor $(\liel,L\cap K)$ and unipotent radical $\lieu$.

We say that a pair $({\mathcal G},{\mathcal L})$ of two full abelian subcategories ${\mathcal G}$ and ${\mathcal L}$ of ${\mathcal C}(\lieg,K)$ resp.\ ${\mathcal C}(\liel,L\cap K)$ is $\lieu$-{\em admissible} if the following three conditions are satisfied:
\begin{itemize}
\item[(i)] The category ${\mathcal G}$ contains the trivial representation and the category ${\mathcal L}$ contains the objects
$$
\bigwedge^q\lieu^*
$$
for all $q\in\ZZ$.
\item[(ii)] The category ${\mathcal L}$ is closed under tensoring with the modules
$$
\bigwedge^q\lieu^*
$$
for any $q\in\ZZ$.
\item[(iii)] The $\lieu$-cohomology of any object in ${\mathcal G}$ lies in ${\mathcal L}$.
\end{itemize}
If the pair of categories satisfies one of the following two conditions, we say that the pair is {\em multiplicative} resp.\ {\em has duality}.
\begin{itemize}
\item[(iv)] If for two objects $V$ and $W$ in ${\mathcal G}$ the tensor product $V\otimes W$ lies in ${\mathcal G}$, then for any $p,q\in\ZZ$
$$
H^p(\lieu;V)\otimes_\CC H^q(\lieu;W)
$$
is an object in ${\mathcal L}$.
\item[(v)] If for $V$ in ${\mathcal G}$ the dual $V^*$ lies in ${\mathcal G}$, then for any $q\in\ZZ$ the objects
$$
H^q(\lieu; V)^*,\;
H^q(\lieu; V'),\;
H^q(\lieu; V'/V^*)
$$
lie in ${\mathcal L}$, where $V'$ denotes the locally $L\cap K$-finite dual, and furthermore
$$
\sum_q(-1)^q[H^q(\lieu; V'/V^*)]\;=\;0
$$
in the Grothendieck group of $\mathcal L$.
\end{itemize}

For convenience of formulations we assume that $\mathcal G$ and $\mathcal L$ each contain each object of the ambient category which is isomorphic to an object of $\mathcal G$ resp.\ $\mathcal L$.

We call a parabolic subalgebra $\lieq\subseteq\lieg$ {\em constructible} if it is germane and there exist parabolic subalgebras
\begin{equation}
\lieq_r=\lieq\;\subseteq\;\lieq_{r-1}\;\subseteq\;\cdots\;\subseteq\;\lieq_1\;\subseteq\;\lieq_0 = \lieg
\label{eq:constructible}
\end{equation}
with Levi decompositions
$$
\lieq_i\;=\;\liel_i+\lieu_i
$$
and the property that for all $0\leq i <r$ the parabolic subalgebra
$$
\lieq_{i+1}\cap\liel_i\;\subseteq\;\liel_i
$$
is a real or $\theta$-stable parabolic subalgebra of the reductive pair $(\liel_i,L_i\cap K)$.

\begin{proposition}\label{prop:fdadmissiblepairs}
Let $\lieq$ be any germane parabolic subalgebra. Then the pair $({\mathcal C}_{\rm fd}(\lieg,K),{\mathcal C}_{\rm fd}(\liel,L\cap K))$ is always $\lieu$-admissible, multiplicative and has duality.
\end{proposition}

\begin{proof}
The $\lieu$-admissibility for $?=\rm fd$ is clear as the standard complex is finite-dimensional in that case. The rest is obvious.
\end{proof}

\begin{proposition}\label{prop:aadmissiblepairs}
Let $\lieq$ be a $\theta$-stable parabolic subalgebra. Then the pair $({\mathcal C}_{\rm a}(\lieg,K),{\mathcal C}_{\rm a}(\liel,L\cap K))$ is always $\lieu$-admissible and has duality.
\end{proposition}

\begin{proof}
The $\lieu$-admissibility follows from Proposition \ref{prop0:admissibility}. In the context of duality we need to distinguish between the $(\lieg,K)$-dual $X^*$ of an admissible $(\lieg,K)$-module $X$ and its $(\lieq,L\cap K)$-dual that we denote $X'$. For any such $X$ we have a natural short exact sequence
$$
\begin{CD}
0@>>> X^*@>\eta>> X'@>>> X^o@>>> 0
\end{CD}
$$
and in general $X^o\neq 0$. However we will show that the $\lieu$-cohomology of $X^o$ vanishes under our hypothesis, which then implies that $\eta$ becomes an isomorphism in cohomology. This is our motivation for the formulation of the duality axiom (v).
The various biduality maps induce a commutative diagram
$$
\begin{CD}
X@>>> {X^*}'@>>> Y\\
@AAA @A\eta'AA @AAA\\
X@>\nu>> X'' @>>> N\\
@AAA @AAA @AAA\\
 0@>>> Z@>>> Z\\
\end{CD}
$$
of $(\lieq,L\cap K)$-modules with short exact sequences in the rows and columns. The exactness of the last row resp.\ last column is a consequence of the snake lemma. As dualizing is exact, we know that $Z={X^o}'$.

To this point we consider the same diagram over the associated compact pairs. Duality is not affected by this and set-theoretically the situation is still the same. The map $\nu$ induces on cohomology by Poincar\'e duality the natural biduality map
$$
H^q(\lieu\cap\liek;X)\;\to\;H^q(\lieu\cap\liek;X)^{**}
$$
which is an isomorphism as admissible modules are reflexive. Hence $\nu_L$ induces an isomorphism
$$
H^q(\lieu\cap\liek;X)\;\to\;H^q(\lieu\cap\liek;X'').
$$
The long exact sequence for the middle row of the above diagram therefore tells us that the $\lieu\cap\liek$-cohomology of $N$ vanishes in all degrees. Hence the long exact sequence for the right most column gives us isomorphisms
$$
H^q(\lieu\cap\liek; Y)\;\to\;H^{q+1}(\lieu\cap\liek; Z).
$$
Together with the general case of Kostant's Theorem \cite{kostant1961}, \cite[Theorem 4.139]{book_knappvogan1995}, this implies the vanishing of the $\lieu\cap\liek$-cohomology of $Y$ and $Z$.

Writing $\lieg=\liep+\liek$ for the Cartan decomposition, the Hochschild-Serre spectral sequence
$$
\bigwedge^p(\lieu\cap\liep)^*\otimes_\CC H^q(\lieu\cap\liek; Z)
\;\Longrightarrow\;H^{p+q}(\lieu; Z).
$$
implies the vanishing of the $\lieu$-cohomology of $Z$. Therefore the long exact sequence for the the middle row short exact sequence of the above commutative diagram provides us with an isomorphism
$$
H^q(\eta'):\;\;\;H^q(\lieu;X'')\;\to\;H^q(\lieu;{X^*}'),
$$
which by Poincar\'e duality is equivalent to the natural map
\begin{equation}
H^q(\eta):\;\;\;H^q(\lieu;X^*)\;\to\;H^q(\lieu;X')
\label{eq:adualiso}
\end{equation}
to be an isomorphism.

For the duality statement our hypothesis guarantees that if $X$ is admissible, then $X^*$ is admissible to and
$$
H^{n-q}(\lieu;X)^*\;\in\;{\mathcal C}_{\rm a}(\liel,L\cap K).
$$
By Poincar\'e duality this implies that
$$
H^q(\lieu;X'\otimes\bigwedge^{\dim\lieu}\lieu)\;\in\;{\mathcal C}_{\rm a}(\liel,L\cap K).
$$
As the dualizing module $\bigwedge^{\dim\lieu}\lieu$ is one-dimensional, we conclude
$$
H^q(\lieu;X')\;\in\;{\mathcal C}_{\rm a}(\liel,L\cap K),
$$
and the isomorphism \eqref{eq:adualiso} concludes the proof.
\end{proof}

\begin{proposition}\label{prop:admissiblepairs}
Let $\lieq$ be a constructible parabolic subalgebra. Then the pairs $({\mathcal C}(\lieg,K),{\mathcal C}(\liel,L\cap K))$, and $({\mathcal C}_{?}(\lieg,K),{\mathcal C}_{?}(\liel,L\cap K))$ for $?\in\{\rm{d,df,fl}\}$ are always $\lieu$-admissible. They are multiplicative for $?=\rm fl$ whenever $\lieq$ is a Borel subalgebra.
\end{proposition}
We conjecture that we have multiplicativity also in the non-minimal cases. We also conjecture that in all cases we have duality. We will show duality for $?=\rm f$ later in the case of a constructible Borel.

\begin{proof}
For $?=\rm fl$ the $\lieu$-admissibility follows by induction over the length $r$ of the filtration \eqref{eq:constructible}. The case $r=0$ is clear. For $r>0$ our induction hypothesis implies that the cohomology $H^q(\lieu_{r-1};X)$ is of finite length as $(\liel_{r-1},L_{r-1}\cap K)$-module for any finite length $(\lieg,K)$-module $X$. As $\lieq\cap\liel_{r-1}$ is $\theta$-stable or real by our assumption, Corollary \ref{cor0:finitelength} and Proposition \ref{prop0:finitelength} imply that the left hand side of the Hochschild-Serre spectral sequence
$$
H^p(\lieu\cap\liel_{r-1};H^q(\lieu_{r-1};X))\;\Longrightarrow\;
H^{p+q}(\lieu;X)
$$
is of finite length. Therefore the right hand side is too, which concludes the induction step.

The $\lieu$-admissibility for $?=\rm d$ follows from the finite length case as cohomology commutes with inductive limits. For $?=\rm df$ it is a bit subtle and results from Proposition \ref{prop0:finiteness} and Corollary \ref{cor0:wigner}, together with the observation that ${\mathcal C}_{\rm df}(\lieg,K)$ is closed under tensorization with finite-dimensional representations, which follows from a Theorem of Kostant \cite{kostant1975}, \cite[Theorem 7.133]{book_knappvogan1995}. This theorem tells us that, given a discretely decomposable $X$ in ${\mathcal C}_{\rm df}(\lieg,K)$ and a finite-dimensional representation $W$, then for a given irreducible $Z$ there are only finitely many composition factors $Z'$ of $X$ such that a given irreducible $Z$ occurs in $Z'\otimes_\CC W$. This shows then $\lieu$-admissibility for $?=\rm df$.

Finally if $(\liel,L\cap K)$ is a Cartan subpair, $(\liel,L\cap K)$-modules of finite length are finite-dimensional. Therefore our $\lieu$-admissible pairs are multiplicative for $?=\rm fl$.
\end{proof}

We write $K({\mathcal G})$ for the Grothendieck group of the category ${\mathcal G}$. If $({\mathcal G},{\mathcal L})$ is $\lieu$-admissible, then by axiom 3 the long exact sequence of cohomology furnishes a group homomorphism
$$
H:K({\mathcal G})\to K({\mathcal L}),
$$ 
$$
[V]\mapsto\sum_{q}(-1)^q[H^q(\lieu;V)],
$$
as $V$ is always of finite cohomological dimension. We define the Weyl denominator
$$
W_\lieq:= H([\CC])=\sum_{q}(-1)^q[\bigwedge^q\lieu^*].
$$
Then the $\lieu$-admissibility guarantees that any element in $K({\mathcal L})$ may be multiplied with $W_\lieq$.

For applications we cannot always and usually don't have to work with the full Grothendieck group in the image. Therefore we introduce the notion of {\em $\lieu$-admissible quadruple}
$$
\mathcal Q=(\mathcal G,\mathcal L, K_{\mathcal G}, K_{\mathcal L})
$$
as follows. $K_{\mathcal G}$ resp.\ $K_{\mathcal L}$ are subgroups of $K(\mathcal G)$ resp.\ $K(\mathcal L)$, subject to the following conditions:
\begin{itemize}
\item[(vi)] $H(K_{\mathcal G})\subseteq K_{\mathcal L}$
\item[(vii)] for every $C\in K_{\mathcal L}$ we have
$$
C\cdot W_\lieq\in K_{\mathcal L}.
$$
\end{itemize}
Furthermore, we say that the quadruple is {\em multiplicative} if the pair $(\mathcal G,\mathcal L)$ is multiplicative and if furthermore for any two elements $X,Y\in K_{\mathcal G}$ such that $X\cdot Y$ is well defined in $K(\mathcal G)$, then $X\cdot Y\in K_{\mathcal G}$. Note that this implies that $H(X)\cdot H(Y)$ is a well defined element of $K_{\mathcal L}$ (cf. proof of Theorem \ref{thm:main1} below).

Similarly we say that the quadruple has {\em duality}, if the pair $(\mathcal G,\mathcal L)$ has duality and if furthermore for any element $X\in K_{\mathcal G}$ of finite length such that $X^*$ is well defined as an element of $K(\mathcal G)$, then $X^*\in K_{\mathcal G}$.

For a $\lieu$-admissible pair, the quadruple $(\mathcal G,\mathcal L,K(\mathcal G),K(\mathcal L))$ is always $\lieu$-admissible and it is multiplicative resp.\ has duality if the pair $(\mathcal G,\mathcal L)$ is multiplicative resp.\ has duality.

For a $\lieu$-admissible quadruple $\mathcal Q$ the localization
$$
C_\lieq({\mathcal Q}):=K_{\mathcal L}[W_\lieq^{-1}]
$$
is well defined. It comes with a canonical group homomorphism $K_{\mathcal L}\to C_\lieq({\mathcal Q})$, which is injective in the finite-dimensional case, but not injective in general. Furthermore the partially defined multiplication of $K({\mathcal L})$ has a natural (but still only partially defined) extension to $C_\lieq({\mathcal L})$, such that $W_\lieq$ is invertible in $C_\lieq({\mathcal L})$ by the very definition of localization.

We define the {\em algebraic character} of $X\in K_{\mathcal G}$ with respect to $\lieq$ as
$$
c_\lieq(X):=H(X)/W_\lieq\in C_\lieq({\mathcal Q}).
$$
The basic properties of $c_\lieq$ are summarized in
\begin{theorem}\label{thm:main1}
The map $c_\lieq$ is a group homomorphism $K_\mathcal G\to C_{\lieq}(\mathcal Q)$, carrying $[\CC]$ to $[\CC]$ (i.e.\ $1$ to $1$).
If the restriction of $X\in K_{\mathcal G}$ to $(\liel,L\cap K)$ lies in $K_{\mathcal L}$ then as an element of $C_{\lieq}(\mathcal Q)$
$$
c_\lieq(X)=[X].
$$
If the quadruple is multiplicative and if for $X,Y\in K_{\mathcal G}$ the product $X\cdot Y$ lies in $K_{\mathcal G}$ as well, then
$$
c_\lieq(X\cdot Y)=c_\lieq(X)\cdot c_\lieq(Y).
$$
If the quadruple has duality and if $X^*\in{\mathcal G}$, then
$$
c_\lieq(X^*)=c_\lieq(X)^*.
$$
\end{theorem}

\begin{proof}
Assume that $V$ lies in $\mathcal L$. The $\lieu$-cohomology of $V$ may be computed from the finite-dimensional complex
$$
\bigwedge^\bullet\lieu^*\otimes_\CC V,
$$
which means that we have the Riemann-Roch formula
$$
H(V)=\sum_{q}(-1)^q[\bigwedge^q\lieu^*]\cdot[V]
$$
in $K(\mathcal L)$. This shows the first identity.

Assume now that $(\mathcal G,\mathcal L)$ is multiplicative. Consider the diagonal embedding of Lie algebras
$$
\Delta:\lieg\to\lieg\times\lieg, g\mapsto (g,g).
$$
We have two full categorical embeddings
$$
i_{1,2}:{\mathcal G}\to{\mathcal G}\times{\mathcal G},
$$
given by
$$
i_1:V\mapsto V\otimes_\CC \CC,
$$
$$
i_2:W\mapsto \CC\otimes_\CC W,
$$
considered as $\lieg\times\lieg$-modules under the standard action on the tensor product
$$
(g_1,g_2)(v\otimes w)=(g_1 v)\otimes w+ v\otimes g_2w.
$$
By the K\"unneth Formula and the Hochschild-Serre spectral sequence for the diagonal empedding $\lieu\to\lieu\times\lieu$, we have in $K(\mathcal L)$
$$
\sum_{p,q}(-1)^{p+q}[\bigwedge^p\lieu^*]\cdot [H^q(\lieu;V\otimes_\CC W)]=
$$
$$
\sum_{p,q}(-1)^{p+q}[H^p(\lieu;V)]\cdot[H^q(\lieu;W)],
$$
and consequently
$$
W_\lieq\cdot H(V\otimes_\CC W)=H(V)\cdot H(W).
$$
By the multiplicativity of $\mathcal Q$, the result follows as desired.

The duality statement follows from Poincar\'e duality. Let $V$ be an object of $\mathcal G$. As already seen before we need to consider the $(\lieq,L\cap K)$-dual $V'$ of $V$ as well. Poincar\'e duality yields by the already proven multiplicativity shows the identity
$$
\sum_q (-1)^q[H^q(\lieq;V')]\;=\;
(-1)^{\dim\lieu}\cdot
[\bigwedge^{\dim\lieu}\lieu^*]\cdot
H(V)^*
$$
in $K(\mathcal L)$. Now the duality axiom (v) tells us that
$$
H(V^*)\;=\;
\sum_q (-1)^q[H^q(\lieq;V')].
$$
This reduces us to the trivial instance
$$
H(\CC)\;=\;
(-1)^{\dim\lieu}\cdot
[\bigwedge^{\dim\lieu}\lieu^*]\cdot
H(\CC)^*
$$
of Poincar\'e duality.
\end{proof}

\begin{proposition}
Let $\mathcal Q$ be a $\lieq$-admissible quadruple, and assume that it satisfies the same admissibility conditions as above with cohomology $H^q$ replaced by homology $H_q$. Then for any $X\in K_{\mathcal G}$
$$
c_\lieq(X)=\frac{\sum_{q}(-1)^q [H_q(\lieu;X)]}{\sum_{q}(-1)^q [H_q(\lieu;\CC)]}.
$$
\end{proposition}

\begin{proof}
This follows from the duality isomorphism \eqref{eq:hardduality}, combined with an analogous argument as in the proof of the duality statement of Theorem \ref{thm:main1}.
\end{proof}

If $\lieq$ is a Borel, i.e.\ if $(\liel,L\cap K)$ is a Cartan subpair, we call $c_\lieq$ the {\em absolute} character. Otherwise we say that $c_\lieq$ is a {\em relative} character. This terminology is justified by
\begin{proposition}\label{prop:ccomposition}
Let $(\liep,M\cap K)\subseteq (\lieq,L\cap K)$ be germane parabolic pairs with nilpotent radicals $\lieu\subseteq \lien$ respectively and let $(\mathcal G,\mathcal L)$ resp.\ $(\mathcal L,\mathcal M)$ be $\lieu$- resp.\ $\lien\cap\liel$-admissible pairs. Assume that $(\mathcal G,\mathcal M)$ is $\lien$-admissible and that the quadruples $\mathcal Q=(\mathcal G,\mathcal L,K_{\mathcal G},K_{\mathcal L})$, $\mathcal Q'=(\mathcal L,\mathcal M,K_{\mathcal L},K_{\mathcal M})$ and $\mathcal Q''=(\mathcal G,\mathcal M,K_{\mathcal G},K_{\mathcal M})$ are $\lieu$- resp.\ $\lien\cap\liel$- resp.\ $\lien$-admissible. Then
$$
c_\liep=c_{\liep\cap\liel} \circ c_\lieq.
$$
\end{proposition}

We remark that the first statement of Theorem \ref{thm:main1} is a special case of this Proposition.

Here $c_{\liep\cap\liel}$ is considered as the character of an $(\liel,L\cap K)$-module with respect to the parabolic subpair induced by $(\liep,M\cap K)$, naturally extended to a map $C_\lieq(\mathcal Q)\to C_\liep(\mathcal Q'')$.

\begin{proof}
Another application of the Hochschild-Serre spectral sequence similar to the proof of Proposition \ref{prop:restriction} below.
\end{proof}

The $\lieu$-admissibility of $(\mathcal G,\mathcal M)$ resp. $\mathcal Q''$ is a weak condition and is nearly automatic. We only need to assume that $\mathcal M$ contains the abutments of the Hochschild-Serre spectral sequences in question, and similarly for the Grothendieck groups. In this sense admissibility is transitive.

In principle this proposition reduces the study of characters of $(\lieg,K)$-modules to the study of relative characters for maximal parabolic subalgebras. More importantly it reduces the study of constructible parabolics to the real and $\theta$-stable cases. We will apply this later.

Let $\iota:(\lieg',K')\to(\lieg,K)$ be an inclusion of reductive pairs, so in particular it is assumed to be compatible with the Cartan involutions and all the other data associated to the pairs. We assume that we are given categories $\mathcal G'$ resp.\ $\mathcal G$ such that forgetting along $\iota$ induces a well defined functor $\mathcal G\to\mathcal G'$.

Let $\lieq'\subseteq\lieg'$ resp.\ $\lieq\subseteq\lieg$ be associated germane parabolic subalgebras satisfying $\lieq'\subseteq\lieq$ with corresponding Levi factors $L'\subseteq L$. Write $\liel'$ resp.\ $\liel$ for the complexified Lie algebras of $L$ and $L'$ and $\lieu'$ resp.\ $\lieu$ for the nilpotent radicals of $\lieq'$ resp.\ $\lieq$.

Let us assume for simplicity additionally $\lieu'\subseteq\lieu$. This is a hypothesis which may be omitted but that would make for a less clean proof and a more complicated definition of the relative Weyl denominator below.

Fix two categories $\mathcal L$ and $\mathcal L'$ such that $(\mathcal G,\mathcal L)$ and $(\mathcal G',\mathcal L')$ are $\lieu$- resp.\ $\lieu'$-admissible and restriction along $\iota$ again gives rise to a well defined functor $\iota:\mathcal L\to\mathcal L'$.

Let $Q=(\mathcal G,\mathcal L, K_{\mathcal G},K_{\mathcal L})$ be $\lieu$- and $Q'=(\mathcal G',\mathcal L', K_{\mathcal G'},K_{\mathcal L'})$ be $\lieu'$-admissible quadruples. We assume that $\iota$ induces well defined maps
$$
\iota:K_{\mathcal G}\to K_{\mathcal G'}
$$
and
$$
\iota:K_{\mathcal L}\to K_{\mathcal L'}.
$$
We define
$$
W_{\lieq/\lieq'}:=
W_{\lieq}|_{\liel',L'\cap K'}/W_{\lieq'},
$$
and assume that it is an element of $K_{\mathcal L'}$, and that the latter is closed under tensorization with $W_{\lieq/\lieq'}$. Then the homorphism
$$
\iota:C_{\lieq}(\mathcal Q)\to C_{\lieq'}(\mathcal {Q}')[W_{\lieq/\lieq'}^{-1}],\;c\mapsto c|_{\liel',L'\cap K'}
$$
is well defined. We denote the corresponding homomorphism $K_{\mathcal G}\to K_{\mathcal G'}$ defined by restriction along $\iota$ the same.
\begin{proposition}\label{prop:restriction}
Let $\iota:(\lieg',K')\to(\lieg,K)$ be an inclusion of reductive pairs as above, in particular mapping one regular element to a regular one. Assume furthermore that the restriction along $\iota$ is compatible with the quadruples $\mathcal Q$ and $\mathcal Q'$ as above.

Then restriction defines an additive and multiplicative map $\iota:C_\lieq(\mathcal Q)\to C_{\lieq'}(\mathcal Q')[W_{\lieq/\lieq'}^{-1}]$, respecting duals, with the property
$$
\iota\circ c_{\lieq}=c_{\lieq'}\circ\iota.
$$
\end{proposition}

\begin{proof}
We have a Hochschild-Serre spectral sequence
$$
E_1^{p,q}=\bigwedge^p(\lieu/\lieu')^*\otimes_\CC H^q(\lieu';X'),
$$
$$
E_1^{p,q}\Longrightarrow H^{p+q}(\lieu; X).
$$
All differentials in this spectral sequence are indeed $(\liel',L'\cap K')$-linear. On the one hand, by our assumption on $\iota$ the cohomology $H^q(\lieu';X')$ lies in $\mathcal L'$, as does $H^{p+q}(\lieu;X)$ and hence we deduce in $K(\mathcal L')$ the identity
$$
\sum_{p}(-1)^p\cdot[\bigwedge^p(\lieu/\lieu')^*]\cdot
\sum_{q}(-1)^q\cdot[H^q(\lieu';X)]=
\sum_{p}(-1)^p\cdot[H^p(\lieu;X)]|_{\liel',L'\cap K'}.
$$
On the other hand, due to our hypothesis
$$
\sum_{p}(-1)^p\cdot[\bigwedge^p(\lieu/\lieu')^*]=W_{\lieq/\lieq'}.
$$
Obviously $\iota$ is multiplicative and respects duals, concluding the proof.
\end{proof}

An immediate application to admissible modules is the comparison between the {\em full} and the {\em compact character}, i.e.\ the character of an admissible $(\lieg,K)$-module $X$ with to a $\theta$-stable parabolic $\lieq$ to its restriction to $K$ or even $K^0$. Both cases are covered by Proposition \ref{prop:restriction}. This shows the relation between characters and generalized Blattner formulae. With the machinery developed in the last section we may deduce Blattner formulae formally from character formulae. We come back to this question in section \ref{sec:blattner} in more detail.

As we have a purely algebraic proof of the Blattner formula for the Zuckerman-Vogan modules $A_\lieq(\lambda)$, due to Zuckerman, it would be interesting to deduce from this the character formulae. At least for the discrete series this actually works, and in some sense, inverts the approach taken by Hecht and Schmid in their proof of the Blattner conjecture for the discrete series \cite{hechtschmid1975}.

\section{Translation functors}

By the theory of the Jantzen-Zuckerman translation functors representations occur in \lq{}coherent families\rq{}. As Vogan has shown (\cite{vogan1979ii}, \cite[Theorem 7.242]{book_knappvogan1995}), $\lieu$-cohomology behaves well under translation functors, which implies that algebraic characters do so as well. We make this precise in this section.

We use the following notation. We do not assume that $K$ is connected nor that $\lieq$ is minimal in the first half of this section. If $\lambda$ is a character of $\lieh$, we write ${\mathcal P}_\lambda$ for the endofunctor ${\mathcal C}_{\rm fl}(\lieg,K)\to{\mathcal C}_{\rm fl}(\lieg,K)$, which projects to the $\lambda$-primary component for the action of $Z(\lieg)$. This functor extends to a functor
$$
{\mathcal P}_\lambda:\mathcal C_{\rm df}(\lieg,K)\to\mathcal C_{\rm fl}(\lieg,K).
$$
By Proposition 7.20 of loc.\ cit., ${\mathcal P}_\lambda$ is exact and hence induces a map on the Grothendieck groups. Let $\lambda'$ be another character of $\lieh$ such that $\mu:=\lambda'-\lambda$ is algebraically integral. Write $F^\mu$ for an irreducible finite-dimensional $(\lieg,K)$-representation of extreme weight $\mu$ which remains irreducible as a $\lieg$-module (and is supposed to exist). Then one can define the {\em translation functor}
$$
\psi_{\lambda,F^\mu}^{\lambda'}
:={\mathcal P}_{\lambda'}\circ(\;\cdot\otimes_\CC F^\mu)\circ{\mathcal P}_\lambda.
$$
This is an exact functor ${\mathcal C}_{\rm df}(\lieg,K)\to{\mathcal C}_{\rm fl}(\lieg,K)$ that descends to the Grothendieck group as well. In our terminology Theorem 7.242 of loc.\ cit.\ reads as follows. Write $E^\mu$ for the irreducible $(\liel,L\cap K)$-submodule of $F^\mu$ containing the weight space for $\mu$.

Suppose that $\lambda$ and $\mu$ satisfy
\begin{enumerate}
\item[(i)] $\lambda+\rho(\lieu)$ and $\lambda+\mu+\rho(\lieu)$ are integrally dominant relative to $\Delta^+(\lieg,\lieh)$,
\item[(ii)] $K$ fixes the $Z(\lieg)$ infinitesimal characters $\chi_{\lambda+\rho(\lieu)}$ and $\chi_{\lambda+\mu+\rho(\lieu)}$,
\item[(iii)] $L\cap K$ fixes the $Z(\liel)$ infinitesimal characters $\chi_{\lambda}$ and $\chi_{\lambda+\mu}$,
\item[(iv)] $\lambda+\rho(\lieu)$ is at least as singular as $\lambda+\mu+\rho(\lieu)$.
\end{enumerate}
Then for any $(\lieg,K)$-module $X$ of finite length, and as a consequence also for any module in $\mathcal C_{\rm df}(\lieg,K)$,
$$
\psi_{\lambda,E^\mu}^{\lambda+\mu}(W_\lieq\cdot c_\lieq(X))=
{\mathcal P}_{\lambda+\mu}(W_\lieq\cdot c_\lieq(\psi_{\lambda+\rho(\lieu),F^\mu}^{\lambda+\mu+\rho(\lieu)}(X))).
$$
Note that on the one hand, the proof of Theorem 7.242 eventually shows more. Namely that
$$
\psi_{w(\lambda+\rho(\lieu))-\rho(\lieu),E^{w(\mu)}}^{w(\lambda+\mu+\rho(\lieu))-\rho(\lieu)}(W_\lieq\cdot c_\lieq(X))=
$$
\begin{equation}
{\mathcal P}_{w(\lambda+\mu+\rho(\lieu))-\rho(\lieu)}(W_\lieq\cdot c_\lieq(\psi_{\lambda+\rho(\lieu),F^\mu}^{\lambda+\mu+\rho(\lieu)}(X)))
\label{eq:ctranslation}
\end{equation}
for any $w\in W(\lieg,\lieh)$.

Assume now that $K$ is connected and $\liel=\lieh$, i.e.\ $\lieq=\liel+\lieu$ is a germane Borel, and that the categories of finite-length modules for $(\lieg,K)$ and $(\liel,L\cap K)$ are $\lieu$-admissible.

In this setting $E^\mu$ is always one-dimensional. Assume furthermore that $X$ has infinitesimal character $\lambda+\rho(\lieu)$. Then the parameter $\lambda$ is essentially translated by $\mu$ (the projections have no effect in this case). Furthermore $F^\mu$ exists always and conditions (ii) and (iii) become vacuous as $K$ is connected. To study the effect on characters we make use of Proposition \ref{prop0:finiteness} which says that
$$
{\mathcal P}_\nu(W_\lieq\cdot c_\lieq(X))\neq 0
$$
implies that $\nu=w(\lambda+\rho(\lieu))-\rho(\lieu)$ for some $w\in W(\lieg,\lieh)$. Applying the twist $E^\mu$ and another projection we arrive at the question when
$$
{\mathcal P}_\nu({\mathcal P}_{w(\lambda+\rho(\lieu))-\rho(\lieu)}(W_\lieq\cdot c_\lieq(X))\cdot[E^\mu])\neq 0.
$$
This amounts to $\nu=w(\lambda+\rho(\lieu))+\mu-\rho(\lieu)$. Consequently we can detect all multiplicities of non-zero $U(\liel)$-isotypic components in $W_\lieq\cdot c_\lieq(X)$ by consideration of
$$
\psi_{w(\lambda+\rho(\lieu))-\rho(\lieu),E^{w(\mu)}}^{w(\lambda+\mu+\rho(\lieu))-\rho(\lieu)}(W_\lieq\cdot c_\lieq(X)),
$$
for all $w\in W(\lieg,\lieh)$.

Now we know by Proposition \ref{prop0:finiteness} that 
\begin{equation}
W_\lieq\cdot c_\lieq(X)=\sum_{w\in W(\lieg,\lieh)}n_w^{\lambda}[\CC_{w(\lambda+\rho(\lieu))-\rho(\lieu)}]
\label{eq:linearcomb}
\end{equation}
for integers $n_w^\lambda\in\ZZ$ and similarly for $\psi_{\lambda,F^\mu}^{\lambda+\mu}(X)$. Note that the coefficients are not uniquely determined in general. Plugging all this together into \eqref{eq:ctranslation} we find

\begin{proposition}
Under the above hypothesis we may assume in \eqref{eq:linearcomb} that
$$
n_w^\lambda=n_w^{\lambda+\mu}
$$
for all $w\in W(\lieg,\lieh)$. In particular the character $c_\lieq(\psi_{\lambda,F^\mu}^{\lambda+\mu}(X))$ uniquely determines the character $c_\lieq(X)$, and in the equisingular case also vice versa.
\end{proposition}

\section{Applications}

\subsection*{Compact groups resp.\ finite-dimensional representations}

We consider the special case where $G$ is compact connected, $(\liel,L)$ is a Cartan subpair and the $\lieu$-admissible pair $(\mathcal G,\mathcal L)$ consists of the categories of finite-dimensional representations. We set $K_{\mathcal G}:=K(\mathcal G)$ and $K_{\mathcal L}:=K(\mathcal L)$ and consider the corresponding $\lieu$-admissible quadruple $\mathcal Q$. Note that in this special situation the canonical map $K_{\mathcal L}\to C_\lieq(\mathcal Q)$ is injective. Theorem \ref{thm:main1} then says that
$$
c_\lieq:K(\mathcal G)\to C_\lieq(\mathcal Q)
$$
is a ring homomorphism which factors over the forgetful map
$$
K(\mathcal G)\to K(\mathcal L).
$$
In particular algebraic characters characterize finite-dimensional representations up to isomorphism by the classical highest weight theory. Furthermore Kostant's Theorem gives

\begin{theorem}[Weyl Character Formula]\label{thm:weyl}
If $G$ is compact connected and $V$ is irreducible of highest weight $\lambda$, then we have
$$
c_\lieq(V)=
\frac
{\sum_{w\in W(\lieg,\liel)}(-1)^{\ell(w)}[\CC_{w(\lambda+\rho(\lieu))-\rho(\lieu)}]}
{\sum_{w\in W(\lieg,\liel)}(-1)^{\ell(w)}[\CC_{w(\rho(\lieu))-\rho(\lieu)}]}.
$$
\end{theorem}

The standard complex for $\lieu$-cohomology resp.\ shows

\begin{proposition}[Weyl Denominator Formula]\label{prop:denominator}
$$
W_\lieq=
{\sum_{w\in W(\lieg,\liel)}(-1)^{\ell(w)}[\CC_{w(\rho(\lieu))-\rho(\lieu)}]}=
\prod_{\alpha\in\Delta^+(\lieg,\liel)}(1-[\CC_{-\alpha}]).
$$
\end{proposition}

\subsection*{Modules of finite length}

In this section $(\lieg,K)$ is a reductive pair associated to a linear\footnote{Linearity is only included for some minor technical reasons in \cite{vogan1979ii}.} real reductive Lie group $G$ in Harish-Chandra's class and we keep the rest of the notation as before. By Proposition \ref{prop:admissiblepairs} we know that for a constructible parabolic subalgebra $\lieq=\liel+\lieu$ of $(\lieg,K)$the pair
$$
(\mathcal G,\mathcal L)\;:=\;(\mathcal C_{\rm fl}(\lieg,K),\mathcal C_{\rm fl}(\liel,L\cap K))
$$
is always $\lieu$-admissible. With $K_{\mathcal G}=K(\mathcal G)$ and $K_{\mathcal L}=K(\mathcal L)$ we get a $\lieu$-admissible quadruple $\mathcal Q$. All identities are to be understood in $C_\lieq(\mathcal Q)$.

\begin{proposition}\label{prop:finiteness}
Assume $K$ to be connected and $X$ be a $(\lieg,K)$-module of finite length. If $X$ has infinitesimal character $\chi$ and if $\lieq$ is a Borel, then there exist integers $n_w\in\ZZ$ for $w\in W(\lieg,\liel)$ such that
$$
W_\lieq\cdot c_\lieq(X)=\sum_{w\in W(\lieg,\liel)}n_w[\CC_{w(\chi)-\rho(\lieu)}]
$$
\end{proposition}

\begin{proof}
This is an immediate consequence of a special case of of Proposition \ref{prop0:finiteness}.
\end{proof}

\begin{proposition}\label{prop:constructiblesexist}
For any reductive pair $(\lieg,K=$ there exist finitely many constructible Borel subalgebras $\lieq_1,\dots,\lieq_r$ whose Levi factors $L_1,\dots,L_r$ have the property that the union of their $G$-conjugates cover a dense subset of $G$.
\end{proposition}

\begin{proof}
By Matsuki \cite{matsuki1979} every $G$-conjugacy class of Borel subalgebras contains a germane Borel subalgebra $\lieq$ and in particular a $\theta$-stable Cartan subalgebra $\lieh$ giving rise to a Cartan pair $(\liel,L\cap K)$ of $(\lieg,K)$. Thus we are reduced to show that every such Cartan pair occurs as the Levi factor of a constructible Borel.

We have the canonical real form $\liel_0\subseteq\lieg_0$ of $\liel$ and decompose it into $\liel_0=\liet_0\oplus\liea_0$, where $\liet_0=\liek_0\cap\liel_0$ and $\liea_0=\liel_0\cap\liep_0$, where $\liep_0$ is the orthogonal complement of $\liek_0$ in $\lieg_0$. Now choose an ordering of the non-zero weights of $\liet$ occuring in $\lieg$, which gives rise to a subset $\Delta^+(\lieg,\liet)\subseteq\Delta(\lieg,\liet)$ of the set of non-zero roots of $\liet$ occuring in $\lieg$. From this we deduce a subset $\Delta^+\subseteq\Delta(\lieg,\liel)$ in the following way. A root $\alpha$ is in $\Delta^+$ if and only if
$$
\alpha|_\liet\in\Delta^+(\lieg,\liet)\cup\{0\}.
$$
Then $\Delta^+$ contains some positive system $\Delta^+(\lieg,\liel)$, is closed under addition in $\Delta(\lieg,\lieh)$, and is $\theta$-stable by definition. Hence it defines a $\theta$-stable parabolic pair $(\lieq',L'\cap K)$ with a Levi pair $(\liel',L'\cap K)$ containing the Cartan pair $(\liel,L\cap K)$.

Now all roots of $\liel$ in $\liel'$ are real by construction, as for any $\alpha\in\Delta(\liel',\liel)$
$$
\alpha|_\liet=0,
$$
which concludes the proof.
\end{proof}

\begin{theorem}\label{thm:main2}
Let $V,W$ be two $(\lieg,K)$-modules of finite length and assume that
$$
c_{\lieq_i}(V)=c_{\lieq_i}(W)
$$
for a set of constructible parabolic subalgebras $\lieq_1,\dots,\lieq_r$ whose Levi factors $L_1,\dots,L_r$ cover a dense subset of $G$ up to conjugation. Then $V$ and $W$ have the same semi-simplifications.
\end{theorem}

\begin{proof}
We show that for each $1\leq i\leq r$ the character $c_{\lieq_i}(V)$ determines the restriction the the global character of $V$ to $L_i$ uniquely. Then by Harish-Chandra's classical results \cite{harishchandra1953,harishchandra1954a,harishchandra1954b}, i.e.\ the regularity and linear independence of characters the claim follows.

By constructibility of the $\lieq_i$ and the transitivity relation from Proposition \ref{prop:ccomposition} this reduces us to the two cases where $\lieq_i$ is real or $\theta$-stable.

If $\lieq_i$ is real, Osborne's Conjecture \cite{osborne1972}, as proven by Hecht and Schmid \cite{hechtschmid1983}, tells us that the restriction of Harish-Chandra's global character $\Theta(V)$ to a \lq{}big\rq{} open subset $U$ of $L_i$, i.e.\ whose conjugates cover the regular elements in $L_i$, and additionally maps surjectively onto the adjoint group of $L_i$, coincides with Harish-Chandra's global character $\Theta(c_{\lieq_i}(V))$, which is formally associated to $c_{\lieq_i}(V)$, restricted to the same set $U$, and $\Theta(V)|_{L_i}$ is uniquely determined by this restriction.

If $\lieq_i$ is $\theta$-stable, Vogan has shown that the analogous statement is true without restricting to a subset of $L_i$ \cite[Theorem 8.1]{vogan1979ii}.
\end{proof}

\section{Algebraic characters and Blattner formulae}\label{sec:blattner}

In this section we use results about the kernels of the localization maps to establish an explicit relation between character formulae and generalized Blattner formulae. The results about localization maps will be established in the last section.

We fix a reductive pair $(\lieg,K)$ and assume that $K$ is connected for simplicity. Furthermore fix a $\theta$-stable parabolic subalgebra $\lieq=\liel+\lieu\subseteq\lieg$, containing a $\theta$-stable Borel subalgebra $\liep=\liet+\lien$ of $\liek$, where $\liet\subseteq\liel$ is the complexified Lie algebra of a maximal torus $T\subseteq L\cap K$ and $\lien$ is the nilpotent radical. We denote $X(T)$ the group of characters of $T$. On $X(T)$ choose an ordering for which the weights occuring in $\liep$ are non-negative.

Choose an abelian category $\mathcal G$ of $(\lieg,K)$-modules and fix an arbitrary weight $\lambda_0\in X(T)$ and define the full subcategory $\mathcal G^{(\lieq)}_{\lambda_0}\subseteq\mathcal G$ consisting of all modules $X$ in $\mathcal G$ with the following property:
\begin{itemize}
\item[(S)] 
For any highest weight $\lambda$ of a $K$-type occuring in $X$, for any $w\in W(K,T)=W(\liek,\liet)$, and any numbering $\beta_1,\dots,\beta_{r}$ of the pairwise distinct elements of $\Delta(\lieu+\lien,\liet)$ and any non-negative integers $n_1,\dots,n_r$ with the property that for any $S\subseteq\{1,\dots,r\}$
\begin{equation}
\sum_{s\in S}n_s\;\leq\;\#\{\beta\in\Delta(\lien,\liet)\mid\exists i\in S:\langle\beta_i,\beta\rangle\neq 0\},
\label{eq:nsum}
\end{equation}
the condition
\begin{equation}
|\langle w(\lambda+\rho(\lien))-\lambda_0,\beta_1\rangle|
\;\geq\;
\frac{1}{2}
\langle \beta_1,\beta_1\rangle
\;+\,
\sum_{i=1}^{r}
\frac{n_{i}+1}{2}
\langle \beta_1,\beta_i\rangle
\label{eq:lambdaregularity}
\end{equation}
is satisfied.
\end{itemize}
In this statement we consider $\Delta(\lieu+\lien,\liet)$ as a set {\em without} multiplicities.


Then $\mathcal G^{(\lieq)}_{\lambda_0}$ is abelian and if $(\mathcal G,\mathcal L)$ is a $\lieq$-admissible pair, then so is $(\mathcal G^{(\lieq)}_{\lambda_0},\mathcal L)$. If $(\mathcal G,\mathcal L)$ is multiplicative or has duality, then $(\mathcal G^{(\lieq)}_{\lambda_0},\mathcal L)$ has the same property.

Let us write $\mathcal K^{(\lieq)}_{\lambda_0}$ for the category of $(\liek,K)$-modules that arises when restricting objects of $\mathcal G^{(\lieq)}_{\lambda_0}$ to $(\liek,K)$. It comes with a surjective faithful functor $\iota:\mathcal G^{(\lieq)}_{\lambda_0}\to\mathcal K^{(\lieq)}_{\lambda_0}$.

Assume that $\mathcal G^{(\lieq)}_{\lambda_0}$ resp.\ $\mathcal L$ are subcategories of $\mathcal C_{\rm f}(\lieg,K)$ resp.\ $\mathcal C_{\rm f}(\liel,L\cap K)$, and write $\mathcal T$ for $\mathcal C_{\rm f}(\liet,T)$. We may assume that objects from $\mathcal L$ restrict to objects in $\mathcal T$, and that $(\mathcal G^{(\lieq)}_{\lambda_0},\mathcal L)$ is $\lieq$-admissible.

Then the pair $(\mathcal K^{(\lieq)}_{\lambda_0},\mathcal T)$ is $\liep$-admissible and we have the diagram
$$
\begin{CD}
K(\mathcal G^{(\lieq)}_{\lambda_0})@>\iota>> K(\mathcal K^{(\lieq)}_{\lambda_0})\\
@Vc_\lieq VV@Vc_\liep VV\\
C_\lieq(\mathcal L)@>\iota>> C_\liep(\mathcal T)[W_{\lieq/\liep}^{-1}]
\end{CD}
$$
which is commutative by Proposition \ref{prop:restriction}. The goal of this section is to show

\begin{theorem}\label{thm:main4}
For any $X$ in $\mathcal G^{(\lieq)}_{\lambda_0}$ its restriction $\iota(X)$ to $K$ is uniquely determined by
$$
\iota(c_{\lieq}(X))\in C_\liep(\mathcal T)[W_{\lieq/\liep}^{-1}]
$$
as a preimage under the map
$$
c_\liep:K(\mathcal K^{(\lieq)}_{\lambda_0})\to C_\liep(\mathcal T)[W_{\lieq/\liep}^{-1}].
$$
\end{theorem}

As mentioned in the introduction, Schmid's bound on $K$-multiplicities in \cite[Theorem 1.3]{schmid1975} enables us to check which discrete series representations satisfy (S), and thus deduce the Blattner Conjecture for these representations by algebraic means.

The idea of proof is to compute the kernel of the map 
$$
c_\liep:K(\mathcal C_{\rm a}(\liek,K))\to C_\liep(\mathcal C_{\rm a}(\liet,T))[W_{\lieq/\liep}^{-1}].
$$
The statement of Theorem \ref{thm:main4} then is equivalent to saying that this kernel intersects $K(\mathcal K^{(\lieq)}_{\lambda_0})$ only trivially. We remark that knowing the kernel implicitly also gives information about the cases violating condition \eqref{eq:lambdaregularity}.

It is crucial for our proof that the multiplicities of $K$-types in finite length modules are asymptotically bounded by the their dimensions, which also implies the existence of global characters. As in the analytic picture, the algebraic picture allows us in principle to relax this boundedness, and to raise the exponent, i.e.\ only require that the multiplicities are asymptotically bounded by (a constant multiple of) the $d$-th power of the norm of the infinitesimal character, where $d$ is fixed for the category of admissible $(\lieg,K)$-modules in question. The choice of $d$ then determines the right hand side of the regularity condition \eqref{eq:lambdaregularity}. It is clear from the proof, and as we will exploit in Section \ref{sec:df} below, the bound, i.e.\ the right hand side of \eqref{eq:lambdaregularity} behaves linearly in $d$.

Finally we point out that our method is universal to any branching problem related to restrictions of reductive pairs in the sense of Proposition \ref{prop:restriction} in the context where the character on the smaller group is absolute. It would be desirable to generalize this approach to the relative case, which in turn may be reduced to the case of a maximal parabolic, again by Proposition \ref{prop:restriction}.

\begin{proof}
First we observe that Harish-Chandra's bound on the multiplicity of the $K$-types in an irreducible $(\lieg,K)$-module $X$ together with the Weyl dimension formula shows that the multiplicities $m_{\lambda+\rho(\lien)}$ of $K$-types in $X$ with highest weight $\lambda$ are bounded by
\begin{equation}
m_{\lambda+\rho(\lien)}\;\;\leq\;\;
C\cdot\!\!\!
\prod_{\beta\in\Delta(\lien,\liet)}\!\!
\langle\lambda+\rho(\lien),\beta\rangle,
\label{eq:weylbound}
\end{equation}
for some constant $C>0$ depending on $X$. In the sequel we use the notation and terminology from section \ref{sec:localization}, with $\liel'=\liet$ and $\lieu'=\lien$. Assume that a non-zero series
$$
m:=\sum_{\lambda} m_{\lambda+\rho(\lien)}\cdot\!\!\!\sum_{w\in W(\liek,\liet)}(-1)^{\ell(w)}
w(\lambda+\rho(\lien))
\;\in\;\CC[[\Lambda^{\frac{1}{2}}]]
$$
maps to zero under the localization map. We denote the coefficient of the monomial $\mu$ in $m$ with $m_\mu$, i.e.\ $m_\mu=(-1)^{\ell(w)}m_{w(\mu)}$, where $w\in W(\liek,\liet)$ is such that $w(\mu)$ is dominant.

As $m$ is in the kernel of the localization map, the series $m$ lies primitively in the kernel of some $t_{\underline{\alpha}}^{\underline{n}}$, and for any $i$ with $n_i\neq 0$ we have
$$
m_i\;:=\;t_{\underline{\alpha}^{\{i\}}}^{\underline{n}^{\{i\}}}\cdot m\;\neq\;0.
$$
Then $m_i$ lies primitively in the kernel of $t_{\alpha_i}^{n_i}$. Assuming $\lambda_0=0$ for a moment, we get by Proposition \ref{prop:kernelt} and Corollary \ref{cor:kernelunibasis} that
$$
m_i\;\in\;
\sum_{k=0}^{n_i}
\CC[[\Lambda^{\frac{1}{2}}]]_{\alpha_i^{\frac{1}{2}}=1}\cdot y_{\alpha_i}^{(k)}
\;+\;
\CC[[\Lambda^{\frac{1}{2}}]]_{\alpha_i^{\frac{1}{2}}=1}\cdot d_{\alpha_i,+}\cdot y_{\alpha}^{(k)}.
$$
By Proposition \ref{prop:yexplicit} the coefficient of $y_{\alpha_i}^{(k)}$ at the monomial $\alpha_i^{\frac{k}{2}+l}$ and of $d_{\alpha_i,+}\cdot y_{\alpha}^{(k)}$ at the monomial $\alpha_i^{\frac{k-1}{2}+l}$ are for all $l\geq 0$ explicitly given by
$$
\binom{k-1+l}{k-1},\;\;\;\text{and}\;\;\;
\frac{k-1+2l}{k-1+l}
\cdot
\binom{k-1+l}{k-1},
$$
respectively. In particular we find an $\eta\in\Lambda^{\frac{1}{2}}$ with the property that the absolute values of the coefficients of $\eta\alpha_i^l$ in $m_i$ grow asymptotically polynomially in $l$ of degree $n_i-1$ for $l\to\infty$.

Now the coefficients $c_\mu$ of $m_i$ for the monomials $\mu$ satisfy
$$
c_\mu=\sum_{u=1}^q x_u\cdot m_{\mu\lambda_u}
$$
for some fixed coefficients $x_u\in\CC$, and fixed weights $\lambda_u\in\Lambda^{\frac{1}{2}}$, all independent of $\mu$. Therefore with the above choice of $\eta$, we find a $1\leq u\leq q$ with the property that the absolute value of $m_{\mu\lambda_u\alpha_i^l}$ grows polynomially in $l$ of degree $n_i-1$.

Then the bound \eqref{eq:weylbound} gives logarithmically
$$
(n_i-1)\cdot\log(l)\;\leq\;\log(C')+\!\!\!\sum_{\beta\in\Delta(\lien,\liet)}\log(\eta,\beta\rangle+l\langle\alpha_i,\beta\rangle)
$$
for some constant $C'$ for all $l$ sufficiently large. In particular this means
$$
n_i\;\leq\;\#\{\beta\in\Delta(\lien,\liet)\mid\langle\alpha_i,\beta\rangle\neq 0\}+1.
$$
Now for any $S\subseteq\{1,\dots,r\}$, replacing the above expression from Proposition \ref{prop:kernelt} with the general formula from Corollary \ref{cor:kernelmultibasis} applied to $t_{\underline{\alpha}^{S}}^{\underline{n}^S}\cdot m$ instead of $m_i$, we deduce along the same lines by considering for suitable $\eta$ the sequence $\eta\prod_{s\in S}\alpha_s^l$ for $l\geq 1$ the inequality
$$
\sum_{s\in S}(n_s-1)\cdot\log(l)\;\leq\;
\log(C')+\!\!\!\sum_{\beta\in\Delta(\lien,\liet)}\log(\langle\eta,\beta\rangle+l\sum_{i\in S}\langle\alpha_i,\beta\rangle).
$$
Hence letting $l\to\infty$ this gives
$$
\sum_{s\in S} n_s\;\leq\;\#\{\beta\in\Delta(\lien,\liet)\mid\exists i\in S:\langle\alpha_i,\beta\rangle\neq 0\}+\#S.
$$

This shows that the regularity condition \eqref{eq:lambdaregularity} then implies condition \eqref{regularity}, which in turn guarantees by Theorem \ref{thm:vanishing} that the kernel of the map $c_\liep$ intersects $K(\mathcal K^{(\lieq)}_{\lambda_0})$ only trivially by Kostant's Theorem, which is true for arbitrary $\lambda_0$.
\end{proof}

\subsection*{An example}

Consider the Blattner formula problem for $(\liegl_3,\SO(3))$, which leads to the non-reduced rank $1$ root system $\pm\alpha,\pm2\alpha$, giving rise to the Weyl denominator
$$
(1-[-\alpha])\cdot
(1-[-2\alpha])^2
$$
once we decree that $\alpha$ be a positive root. We write $Z_{\lambda}$ for an irreducible representation of $\SO(3)$ with hightest weight $\lambda\in\ZZ\cdot\alpha$. Using the explicit description of the kernel of the localization map given in section \ref{sec:localization}, we deduce structure of the kernel of the map
$$
c_{\liep}:K(\SO(3))\;\to\;C_{\liep}({\mathcal T})[W_{\lieq/\liep}^{-1}],
$$
where $K(\SO(3))$ is the subgroup of $K(\mathcal C_{\rm a}(\lieso_3,\SO(3)))$ satisfying Harish-Chandra's bound as follows. Corollary \ref{cor:kernelunibasis} tells us that the kernel of multiplication with $(1-[-\alpha])^n$ is given by
$$
\sum_{k=0}^n
\CC\cdot y_{\alpha}^{(k)}
\;+\;
\CC\cdot d_{\alpha,+}\cdot y_{\alpha}^{(k)}
$$
in the notation of section \ref{sec:localization}. Similarly the kernel of $(1-[-2\alpha])^m$ is given by
$$
\sum_{k=0}^m
\CC\cdot y_{2\alpha}^{(k)}
\;+\;
\CC\cdot \alpha^{\frac{1}{2}}\cdot y_{2\alpha}^{(k)}
\;+\;
\CC\cdot d_{2\alpha,+}\cdot y_{2\alpha}^{(k)}
\;+\;
\CC\cdot \alpha^{\frac{1}{2}}\cdot d_{2\alpha,+}\cdot y_{2\alpha}^{(k)},
$$
where the terms occur as a consequence of the choice $\{1,\alpha^{\frac{1}{2}}\}$ as a system of representatives for $\CC[[\Lambda^{\frac{1}{2}}]]_{\alpha=1}$.

More generally the sum of the kernels of the products $(1-[-\alpha])^n(1-[-2\alpha])^m$ are generated by the collection of all elements
$$
d_{\alpha,+}^b\cdot d_{2\alpha,+}^c\cdot
(s_\alpha^n\cdot s_{2\alpha}^m+(-1)^{m+n+1}
s_{-\alpha}^n\cdot s_{-2\alpha}^m)
$$
where $b,c\in\{0,1\}$, thanks to Proposition \ref{prop:kernelmultibasis}. We emphasize that we assume $n>0$, and therefore the contribution from the system of representatives $\CC[[\Lambda^{\frac{1}{2}}]]_{\alpha=1}$ is trivial, and by the analogous argument of the proof of Proposition \ref{prop:alpharelation} we know that we may also assume $c=0$.

Now Harish-Chandra's bound implies that the coefficients contributing to the kernel grow at most linearly in the degree of $\alpha$. In particular by Proposition \ref{prop:yexplicit} we conclude that $0\leq n,m$ and $n+m\leq 2$, which leaves us with a handful of cases.

It makes sense to choose $\lambda_0=-\frac{1}{2}\alpha$, as Weyl's character formula tells us that the numerator of the character of a representation of $\SO(3)$, shifted by $\lambda_0$, becomes an eigen vector for the eigen value $-1$ of the action of the non-trivial Weyl group element $\tau:\alpha\mapsto-\alpha$.

Then we need to consider the projection of the above generators onto the $(-1)$-eigen space of $\tau$. The projection of the terms
$$
\alpha^{\frac{1}{2}}\cdot
d_{2\alpha,+}^c\cdot
y_{2\alpha}^{(k)},\;\;\;
c\in\{0,1\},
$$
onto the $(+1)$-eigen space is equivalent to multiplication with $d_{\alpha,-}$ whenever $k$ is odd, and equivalent to multiplication with $d_{\alpha,+}$ whenever $k$ is even.

Among the other cases only the cases where $k$ is even are to be taken into consideration.

Finally we know that our half integral shift with $-\lambda_0$ yields only half integral weights, which means that this controls the occurence of the factor $d_{\alpha,+}$.

This leaves us with the four generators
$$
d_{\alpha,+}\cdot y_{\alpha}^{(2)}\;=\;
\sum_{k=0}^\infty (1+2k)(\alpha^{\frac{1}{2}+k}-\alpha^{-\frac{1}{2}-k}),
$$
$$
d_{\alpha,-}\cdot y_{2\alpha}^{(1)}\;=\;
\sum_{k=0}^\infty (-1)^k\cdot (
\alpha^{\frac{1}{2}+k}-\alpha^{-\frac{1}{2}-k}
),
$$
$$
d_{\alpha,+}\cdot y_{2\alpha}^{(2)}\;=\;
\sum_{k=0}^\infty (k+1)((
\alpha^{\frac{3}{2}+2k}-\alpha^{-\frac{3}{2}-2k})+
(\alpha^{\frac{5}{2}+2k}-\alpha^{-\frac{5}{2}-2k})
),
$$
$$
y_{(\alpha,2\alpha)}^{(1,1)}\;=\;
s_\alpha\cdot s_{2\alpha}-
s_{-\alpha}\cdot s_{-2\alpha}\;=\;
\sum_{k=0}^\infty \left\lfloor\frac{k}{2}+1\right\rfloor\cdot
(\alpha^{\frac{3}{2}+k}-\alpha^{-\frac{3}{2}-k}).
$$
We see that the these generators are linearly dependent and eventually the first and the last generate the kernel. Therefore in this case we obtain the sharper result that each element of the kernel satisfying Harish-Chandra's bound contains at least one of the $\SO(3)$-types $Z_{0\cdot\alpha}$ or $Z_{\alpha}$. Explicitly the kernel is generated by the two elements
$$
\kappa_1\;:=\;\sum_{k=0}^\infty (1+2k)[Z_{k\cdot\alpha}],
$$
$$
\kappa_2\;:=\;\sum_{k=0}^\infty \left\lfloor\frac{k}{2}+1\right\rfloor\cdot[Z_{(1+k)\cdot\alpha}].
$$

We conclude that whenever $Z$ is an irreducible $(\liegl_3,\SO(3))$-module with vanishing $\iota(c_\lieu(Z))$, then $Z$ contains one of the $\SO(3)$-types $Z_{a\cdot\alpha}$ with $0\leq a\leq 1$, and the Blattner formula is a linear combination of $\kappa_1$ and $\kappa_2$, subject to the condition that the minimal $\SO(3)$-type occurs with multiplicity one.

We remark that already the condition $a\geq 2$ guarantees disjointness with the kernel and the stronger bound $a\geq 3$ predicted by condition (S) is eventually satisfied by the non-degenerate cohomological irreducible $(\liegl_3,\SO(3))$-modules.

In higher rank cases the kernel is usually infinitely generated, yet there are certain restrictions when it comes to non-virtual representations, which are related to the action of the Weyl group of $K$. We may project onto the $(-1)^{\ell(\cdot)}$-isotypic subspace of this action, and the result is the kernel of the localization map
$$
K({\mathcal T})\to C_{\liep}({\mathcal T})[W_{\lieq/\liep}^{-1}]
$$
coming from virtual representations of $K$, which is obivously infinitely generated, even in the rank $1$ case, if we do not impose Harish-Chandra's bound. After imposing it in the higher rank setting, there might still be possibly infinitely many generators left. This stems from the fact that whenever an element $y$ lies in the kernel $t_{\underline{\alpha}}^{\underline{n}}$, there might be a weight $\beta$ orthogonal to all the roots $\alpha_i$ occuring in $\alpha$, and in particular it may be orthogonal to all the weights on which $y$ is supported. Consequently any Laurent series $f$ in $\beta$ yields new elements $f\cdot y$, the latter being well defined. At the same time multiplication with a power of $\beta$ may have the effect that even if the Weyl orbit of $y$ corresponded to a non-virtual representation, the orbit of $f\cdot y$ may not (or vice versa), as in the corresponding same Weyl orbits opposite signs may occur simultaneously.

\section{Discretely decomposable modules}\label{sec:df}

In this section we generalize the results from the previous section to discretely decomposable modules, again postponing the treatment of localizations to the section \ref{sec:localization}. We assume that we are in the setting as in Proposition \ref{prop:restriction}, and we use the notation as introduced there, subject to the following restrictions.

We assume that we are given a collection of parabolic algebras $\lieq_1,\dots,\lieq_s\subseteq\lieg$ with the property that their intersections $\lieq_1',\dots\lieq_s'$ with $\lieg'$ are constructible Borel algebras in $\lieg'$ with the property that the $G'$-conjugates of their Levi factors $L_1',\dots,L_s'$ cover a dense subset of $G'$. We take the $\lieu_i'$-admissible categories as the pair
$$
(\mathcal G',\mathcal L_i')=
(\mathcal C_{{\rm df}(b)}(\lieg',K')
,\mathcal C_{\rm df}(\liel_i',L_i'\cap K'))
$$
of discretely decomposable modules with finite multiplicities bounded by the exponent $b\geq 0$, i.e.\ $\mathcal C_{{\rm df}(b)}(\lieg',K')$ denotes the full subcategory of $\mathcal C_{\rm df}(\lieg',K')$ consisting of modules $X$ with the property that there are constants $c_X, d_X\geq 0$ such that for {\em all} $1\leq i\leq s$ and any character $\lambda\in{\liel_i'}^*$ we have for each degree $q$ the multiplicity bound
\begin{equation}
m_\lambda(H^q(\lieu_i';X))\;\leq\; c_X\cdot\!\!\!\!\prod_{\beta\in\Delta(\lieu',\liel')}\!\!\!\langle\lambda,\beta\rangle^b + d_X.
\label{eq:multiplicitybound}
\end{equation}
Note that Proposition \ref{prop0:finiteness} tells us that for each $\lambda$ there are only finitely many infinitesimal characters of irreducibles for $(\lieg',K')$ contributing to the above multiplicity. Furthermore this number is bounded by the order of the Weyl group, independently of $\lambda$.

We remark that we may let $b$ vary with $\beta$ and $i$ and get a finer statement. As this only complicates the formulae but not the arguments, we content us to the treatment of a constant $b$ for all $\beta\in\Delta(\lieu_i',\liel_i')$.

On the side of $(\lieg,K)$ and $(\liel_i,L_i\cap K)$ we suppose that we are given $\lieu_i$-admissible categories $(\mathcal G,\mathcal L_i)$ (for example modules of finite length as $\mathcal G$) with the property that their restrictions are in $\mathcal G'$ resp.\ $\mathcal L_i'$.

Now fix a character $\lambda_0\in X(\liel_i')$ and define the full subcategory $\mathcal G^{(b)}_{\lambda_0}\subseteq\mathcal G$ consisting of all modules $X$ in $\mathcal G$ with the following property:
\begin{itemize}
\item[(S')] 
For any irreducible $Z$ with infinitesimal character $\lambda$ occuring in the restriction of $X$ to $(\lieg',K')$, for any $1\leq i\leq s$, any $w\in W(\lieg',\liel_i')$, any numbering $\beta_1,\dots,\beta_r$ of the pairwise distinct elements of the set $\Delta(\lieu_i,\liel_i')$ and any non-negative integers $n_1,\dots,n_r$ with
\begin{equation}
\sum_{s\in S}n_s\;\leq\;b\cdot\#\{\beta\in\Delta(\lieu_i',\liel_i')\mid\exists j\in S:\langle\beta_j,\beta\rangle\neq 0\},
\label{eq:nsumb}
\end{equation}
the condition
\begin{equation}
|\langle w(\lambda)-\lambda_0,\beta_1\rangle|
\;\geq\;
\frac{1}{2}
\langle \beta_1,\beta_1\rangle
+
\sum_{j=1}^r
\frac{n_j+1}{2}
\langle \beta_1,\beta_j\rangle
\label{eq:lambdaregularitydf}
\end{equation}
is satisfied.
\end{itemize}
Again we consider $\Delta(\lieu_i',\liel_i')$ here as a set {\em without} multiplicities. We define for each $1\leq i\leq s$ a $\lieu_i'$-admissible quadruples $\mathcal Q_i$ as in the previous section, in particular ${\mathcal G'_{\lambda_0}}^{\!\!\!\!(b)}$ is the essential image of $\mathcal G^{(b)}_{\lambda_0}\subseteq\mathcal G$ under the restriction map.

\begin{theorem}\label{thm:main5}
For any $X$ in $\mathcal G^{(b)}_{\lambda_0}$ the multiplicity of any composition factor $Z$ of its restriction $\iota(X)$ to $(\lieg',K')$ is uniquely determined by the collection of
\begin{equation}
\iota(c_{\lieq_i}(X))\in C_{\lieq_i'}(\mathcal L_i')[W_{\lieq_i/\lieq_i'}^{-1}]
\label{eq:restrictedchar}
\end{equation}
for $1\leq i\leq s$. In other words the semisimplification of a simultaneous preimage of \eqref{eq:restrictedchar} under the canonical maps
$$
c_{\lieq_i'}:K({\mathcal G'_{\lambda_0}}^{\!\!\!\!(b)})\to C_{\lieq_i'}(\mathcal L_i')[W_{\lieq_i/\lieq_i'}^{-1}]
$$
is uniquely determined.
\end{theorem}

\begin{proof}
Proposition \ref{prop0:finiteness} tells us that to determine the multiplicity of a given $Z$ with infinitesimal character $\lambda_i$ with respect to $\liel_i'$ there are only finitely many $\liel_i$-eigen spaces in $\lieu_i'$-cohomology that are to be considered. Now the total number of isomorphism classes of irreducible $(\lieg',K')$-modules $Z'$ contributing to these eigen spaces and their corresponding Weyl conjugates in the sense of Proposition \ref{prop0:finiteness} is a finite number by Harish-Chandra.

Therefore we are reduced to the case of finite length modules, i.e.\ to Theorem \ref{thm:main2}, once we show that the collection of characters \eqref{eq:restrictedchar} for each $1\leq i\leq s$ determines the Euler characteristic of $\lieu_i'$-cohomology that may contain a contribution from the modules $Z'$.

This is a purely algebraic problem and the argument goes mutatis mutandis along the same lines as the proof of Theorem \ref{thm:main4}, replacing Harish-Chandra's bound with the bound \eqref{eq:multiplicitybound} given by $b$. We deduce that the same formula holds with condition \eqref{eq:nsumb} replacing condition \eqref{eq:nsum} by Theorem \ref{thm:vanishing}.
\end{proof}

In conjunction with Proposition \ref{prop:restriction}, Theorem \ref{thm:main2}, and corresponding character formulae on $(\lieg,K)$ (cf.\ \cite[Theorem 7]{harishchandra1954b} and \cite{hirai1973} for example) this reduces certain branching problems to the question if the restriction in question lies in the category $\mathcal C_{{\rm df}(b)}(\lieg',K')$ and supplemental information about the multiplicities violating \eqref{eq:lambdaregularitydf}.

Kobayashi \cite{kobayashi2000} gives an overview of branching with respect to restriction to a reductive subpair. In particular he formulates the conjecture that if $(G,G')$ is a reductive symmetric pair and whenever an irreducible unitary representation $X$ of $G$ decomposes discretely infinitesimally when restricted to $G'$, then the multiplicities are finite (Conjecture C of loc.\ cit.). Theorem \ref{thm:main5} is our main motivation in formulating Conjecture \ref{conj:polybound}.

For Zuckerman-Vogan's cohomologically induced \lq{}standard modules\rq{} $A_\lieq(\lambda)$ Kobayashi gives a necessary and sufficient criterion when the restriction with respect to a reductive symmetric pair $(G,G')$ decomposes discretely with finite multiplicities \cite{kobayashi1994,kobayashi1997,kobayashi1998}.

In principle the finite multiplicity condition might be relaxed and only required for a carefully chosen subset of all composition factors. Then these finite multiplicities may still be determined by the characters. However in view of Kobayashi's Conjecture C of loc.\ cit., such a treatment should not be necessary, at least in the discretely decomposable case.

\section{Localizations of Grothendieck groups}\label{sec:localization}

We use the notation of section \ref{sec:df}, in particular we are given an inculsion $(\lieg',K')\to(\lieg,\liel)$ of reductive pairs compatible with two Borel subpairs $(\lieq',L'\cap K')\subseteq(\lieq,L\cap K)$ with respective decompositions $\lieq'=\liel'+\lieu'$ and $\lieq=\liel+\lieu$. In particular the Levi factors are Cartan subalgebras.
$$
\Lambda\subseteq \liel'^*
$$
denote the sublattice generated by the elements of $\Delta(\lieu,\liel')$. We consider the $\CC$-vector space of formal unbounded Laurent series $\CC[[\Lambda]]$, i.e.\ we allow arbitrary (not necessarily finite) linear combinations of elements of $\Lambda$. Then $\CC[[\Lambda]]$ is no more a ring, but it is a $\CC[\Lambda]$-module.

We let $\alpha_1,\dots,\alpha_d$ be elements in $\Lambda$, giving rise to a basis of $\Lambda$, containing the simple roots of $\liel'$ in $\lieu'$, and the property that any $\alpha\in\Delta(\lieu,\liel')$ is a sum of some of the elements $\alpha_1,\dots,\alpha_d$.

For notational convenience we write $W_\Lambda$ for the generalized \lq{}Weyl group\rq{} of $\Delta(\lieu,\liel')\cup\Delta(\lieu^-,\liel')$, which is generated by the reflections $w_\alpha$ for $\alpha\in\Delta(\lieu,\liel')$, which in turn are induced by true root reflections of $\lieg$, and are easily seen to be independent of the chosen preimage in $\liel^*$. Then the Weyl group $W(\lieg',\liel')$ is a subgroup of $W_\Lambda$ and the latter acts on $\CC[\Lambda]$ and $\CC[[\Lambda]]$.

We assume that we are given a short exact sequence
$$
0\to\Lambda\to\Lambda^{\frac{1}{2}}\to (\ZZ/2\ZZ)^d\to 0
$$
with a free $\ZZ$-module $\Lambda^{\frac{1}{2}}$ of rank $d$ containing $\Lambda$, generated by square roots of elements in $\Lambda$. For any $\alpha\in\Delta(\lieu,\liel')$ there is a square root
$$
\alpha^\frac{1}{2}\in \Lambda^{\frac{1}{2}}
$$
in the sense that it be a preimage under the map
$$
\Lambda^{\frac{1}{2}}\to \Lambda
$$
given by squaring. We write $\CC[[\Lambda^{\frac{1}{2}}]]$ for the analogously defined $\CC[\Lambda^{\frac{1}{2}}]$-module of Laurent series.

It comes with two actions of two groups. First of all the action of the Weyl group $W_\Lambda$ naturally extends to the above modules and rings. Furthermore the Galois group
$$
G_\Lambda\cong\{\pm 1\}^r
$$
of the inclusion $\CC[\Lambda]\to\CC[[\Lambda^{\frac{1}{2}}]]$ also acts naturally on the above modules and rings. We think of an element $\sigma\in G_\Lambda$ as a collection of signs, i.e.\
$$
\sigma(\alpha)=\sigma_\alpha\cdot \alpha,
$$
where
$$
\sigma_\alpha\in\{\pm 1\}.
$$
Then $G_\Lambda$ is generated by the {\em simple signs} $\sigma_1,\dots,\sigma_d$ with the property that
$$
\sigma_i(\alpha_j^{\frac{1}{2}})=(1-2\delta_{ij})\cdot\alpha_j^{\frac{1}{2}}
$$
for the Kronecker symbol $\delta_{ij}$. The actions of the groups $W_\Lambda$ and $G_\Lambda$ commute.

We fix a $W_\Lambda$-invariant scalar product $\langle\cdot,\cdot\rangle$ on $\Lambda^{\frac{1}{2}}\otimes_\ZZ\RR$ which we assume being induced from a $W(\lieg,\liel)$-invariant scalar product, and tend to write the arguments additively whenever only roots are invlved. We consider the group $\Lambda^{\frac{1}{2}}$ always multiplicatively when it comes to Laurent series.

Consider the elements
$$
d_{\alpha,\pm}:=\alpha^{-\frac{1}{2}}\pm\alpha^{\frac{1}{2}}\in\CC[\Lambda^{\frac{1}{2}}],
$$
$$
s_{\alpha}:=
\alpha^{\frac{1}{2}}
\sum_{k=0}^\infty 
\alpha^{k}
\in\CC[[\Lambda^{\frac{1}{2}}]],
$$
and for $n\geq 0$ we set
$$
y_{\alpha}^{(n)}:=s_{\alpha}^n+(-1)^{n+1} w_\alpha s_{\alpha}^n.
$$
Then
$$
d_{\alpha,-}\cdot s_{\alpha}=1
$$
and
$$
d_{\alpha,-}\cdot w_\alpha s_{\alpha}= -1.
$$
Therefore, for $n>0$
\begin{equation}
d_{\alpha,-}\cdot y_{\alpha}^{(n)}=y_{\alpha}^{(n-1)},
\label{eq:dnpreimage}
\end{equation}
and we conclude that
\begin{equation}
d_{\alpha,-}^n\cdot y_{\alpha}^{(n)}=y_{\alpha}^{(0)}=0.
\label{eq:dnyn}
\end{equation}
For any element $\alpha\in\Delta(\lieu,\liel')$ we define a map
$$
p_{\alpha}:\CC[[\Lambda]]\to \CC[[\Lambda]]
$$
as follows. Given an element
$$
m\in \CC[[\Lambda^{\frac{1}{2}}]],
$$
we have a unique decomposition
$$
m=m_\alpha^++m_\alpha^-
$$
where
$$
m_\alpha^\pm=\sum_{\mu\in\Lambda^{\frac{1}{2}}} c_\mu^\pm \cdot\mu
$$
with $c_\mu\in\CC$ and
$$
c_\mu^\pm=0
$$
if $\mu^2\in\Lambda$ does not lie in the closure of a generalized Weyl chamber, whose closure contains $\pm\alpha$ (i.e.\ $\pm\langle\mu,\alpha\rangle<0$). For those $\mu$ contained in the intersection of the closures of generalized Weyl chambers of $\alpha$ and $-\alpha$ (i.e.\ $\langle\mu,\alpha\rangle=0$), we insist on
$$
c_\mu^-=0,
$$
in order to make an explicit choice.

Then the elements $m_\alpha^\pm$ are uniquely determined by $m$ and $\alpha$ and furthermore
$$
p_{\alpha}(m):=s_{\alpha}\cdot m_\alpha^+-(w_\alpha s_{\alpha})\cdot m_\alpha^-
$$
is well defined as well.
\begin{proposition}
The map $p_{\alpha}$ is a section of the multiplication by $d_{\alpha,-}$ map
$$
t_{\alpha}:\CC[[\Lambda^{\frac{1}{2}}]]\to \CC[[\Lambda^{\frac{1}{2}}]],\;\;\;
f\mapsto d_{\alpha,-}\cdot f.
$$
In particular the latter is surjective.
\end{proposition}

\begin{proof}
We have for any $m\in \CC[[\Lambda^{\frac{1}{2}}]]$
$$
d_{\alpha,-} \cdot p_{\alpha}(m)=
(d_{\alpha,-} s_{\alpha})\cdot m_\alpha^+-(d_{\alpha,-} w_\alpha s_{\alpha})\cdot m_\alpha^-=
1\cdot m_\alpha^+-(-1)\cdot m_\alpha^-=m,
$$
showing the claim.
\end{proof}

For any elements $\beta_1,\dots,\beta_s\in\Lambda $ we denote by $\CC[[\Lambda^{\frac{1}{2}}]]_{(\beta_1,\dots,\beta_s)}$ the subspace of $\CC[[\Lambda^{\frac{1}{2}}]]$ which consists of {\em $(\beta_1,\dots,\beta_s)$-finite} Laurent series in the following sense:
$$
f=\sum_{\lambda\in\Lambda}f_\lambda\lambda
$$
is {\em $(\beta_1,\dots,\beta_s)$-finite}, if for any $\lambda\in\Lambda^{\frac{1}{2}}$ the set
$$
\{(k_1,\dots,k_s)\in\ZZ^s\mid f_{\lambda\beta_1^{k_1}\cdots\beta_s^{k_s}}\neq 0\}
$$
is finite. We have
$$
\CC[\Lambda^{\frac{1}{2}}]=\CC[[\Lambda^{\frac{1}{2}}]]_{(\alpha_1,\dots,\alpha_r)}.
$$

\begin{lemma}\label{lem:pfinitecommute}
For any $n\geq 0$ we have
$$
p_{\alpha}^n\left( \CC[[\Lambda^{\frac{1}{2}}]]_{(\alpha)}\cdot y_{\alpha}^{(1)}\right)\subseteq
\CC[[\Lambda^{\frac{1}{2}}]]_{(\alpha)}\cdot y_{\alpha}^{(n+1)}.
$$
\end{lemma}

\begin{proof}
We easily reduce to the case of a monomial. We have for any $\mu\in\Lambda^{\frac{1}{2}}$
$$
\mu\cdot y_{\alpha}^{(1)}=
\alpha^{\frac{1}{2}}\sum_{k\in\ZZ}\alpha^k\mu,
$$
and as multiplication by $\mu$ translates the Weyl chambers, there is a $k_\mu\in\ZZ$ with the property that
$$
\left(\mu\cdot y_{\alpha}^{(1)}\right)^{\pm}=
\mu\alpha^{\pm\frac{1}{2}}\sum_{\pm k\geq \pm k_\mu}\alpha^{k},
$$
or with the Kronecker symbol $\delta_{1,\pm 1}$
$$
\left(\mu\cdot y_{\alpha}^{(1)}\right)^{\pm}=
\delta_{1,\pm1}\cdot\mu\alpha^{\frac{1}{2}+k_\mu}
+
\mu\alpha^{\pm\frac{1}{2}}\sum_{\pm k> \pm k_\mu}\alpha^{k}.
$$
We assume that we are in the first case, the argument for the second being same. With this notation
$$
p_{\alpha}^n\left( \mu\cdot y_{\alpha}^{(1)}\right)=
s_{\alpha}^n\cdot\left( \mu\cdot y_{\alpha}^{(1)}\right)^+
+
(-1)^n
(w_\alpha s_{\alpha})^n\cdot\left( \mu\cdot y_{\alpha}^{(1)}\right)^-=
$$
$$
\mu\cdot\alpha^{k_\mu}\cdot\left(
s_{\alpha}^{n+1}
+(-1)^n
(w_\alpha s_{\alpha})^{n+1}
\right)=
\mu\cdot\alpha^{k_\mu}\cdot y_{\alpha}^{(n+1)}.
$$
This proves the claim.
\end{proof}

\begin{proposition}\label{prop:kernelt}
For any $\alpha\in\Delta(\lieu,\liel')$, and for any $n\geq 0$ the kernel of $t_{\alpha}^n$ is given by
$$
\CC[[\Lambda^{\frac{1}{2}}]]_{(\alpha)}\cdot y_{\alpha}^{(n)}.
$$
\end{proposition}

\begin{proof}
By relation \eqref{eq:dnyn} the kernel contains the subspace
$$
\CC[[\Lambda^{\frac{1}{2}}]]_{(\alpha)}\cdot y_{\alpha}^{(n)}.
$$
We show by induction on $n$, that this space indeed contains the kernel. The case $n=0$ is clear. For the case $n=1$ we need to show that any element
$$
f=\sum_{\lambda\in\Lambda}f_\lambda\lambda\in\kernel t_{\alpha}
$$
is of the shape
$$
f=g\cdot y_{\alpha}^{(n)}
$$
with
$$
g=
\sum_{\lambda\in\Lambda}
g_\lambda\lambda
\in\CC[[\Lambda^{\frac{1}{2}}]]_{(\alpha)}.
$$
We pick a set of representatives $L_\alpha\subseteq\Lambda^{\frac{1}{2}}$ for $\Lambda^{\frac{1}{2}}/\alpha^\ZZ$, and set for any $\lambda_0\in L_\alpha$ and any $k\in\ZZ$
$$
g_{\lambda_0\alpha^k}:=
\begin{cases}
f_{\lambda_0\alpha^{-\frac{1}{2}}}\;&\text{if}\;k= 0,\\
0\;&\text{if}\;k\neq 0.
\end{cases}
$$
Then
$$
g\cdot y_{\alpha}^{(n)}=
\alpha^{\frac{1}{2}}\cdot\sum_{\lambda_0\in L_\alpha,k,l\in\ZZ}
g_{\lambda_0\alpha^k}\cdot \lambda_0\alpha^k\cdot\alpha^l=
$$
$$
\alpha^{\frac{1}{2}}\cdot\sum_{\lambda_0\in L_\alpha,l\in\ZZ}
f_{\lambda_0\alpha^{-\frac{1}{2}}}\cdot \lambda_0\alpha^l=
\sum_{\lambda_0\in L,l\in\ZZ}
f_{\lambda_0\alpha^{l}}\cdot \lambda_0\alpha^l=f,
$$
because saying that $f$ is annihilated by $d_{\alpha,-}$ is the same to say that
$$
f_{\lambda\alpha^{-\frac{1}{2}}}= f_{\lambda\alpha^{\frac{1}{2}}}.
$$
By construction $g$ is $(\alpha)$-finite and the claim for $n=1$ follows. Now assume that the claim is true for a given $n\geq 1$. We have
$$
\kernel t_{\alpha}^{n+1}=\kernel t_{\alpha}^n+p_{\alpha}^n\kernel t_{\alpha}.
$$
By Lemma \ref{lem:pfinitecommute} we get
$$
p_{\alpha}^n\kernel t_{\alpha}=
p_{\alpha}^n \CC[[\Lambda^{\frac{1}{2}}]]_{(\alpha)}\cdot y_{\alpha}^{(1)}\subseteq
\CC[[\Lambda^{\frac{1}{2}}]]_{(\alpha)}\cdot y_{\alpha}^{(n+1)}.
$$
As
$$
d_{\alpha,-}\in\CC[[\Lambda^{\frac{1}{2}}]]_{(\alpha)}
$$
we see with \eqref{eq:dnpreimage} that
$$
\CC[[\Lambda^{\frac{1}{2}}]]_{(\alpha)}\cdot y_{\alpha}^{(n)}\subseteq
\CC[[\Lambda^{\frac{1}{2}}]]_{(\alpha)}\cdot y_{\alpha}^{(n+1)},
$$
and we already know that the left hand side is the kernel of $t_{\alpha}^n$, concluding the proof.
\end{proof}

For $m\geq 0$ we set
$$
d_{\alpha,\pm}^{(m)}:=\alpha^{-\frac{m}{2}}\pm\alpha^{\frac{m}{2}}
\;\in\;
\CC[[\Lambda]]_{(\alpha)}.
$$
Then the collection of elements
$$
d_{\alpha,+}^{(0)},
d_{\alpha,\pm}^{(1)},
d_{\alpha,\pm}^{(2)},\cdots
$$
forms a $\CC$-basis of $\CC[\alpha^{\frac{1}{2}},\alpha^{-\frac{1}{2}}]$. We have the relations
\begin{equation}
d_{\alpha,+}^{(m)}=
\frac{1}{2}d_{\alpha,-}^{(1)}\cdot d_{\alpha,-}^{(m-1)}+\frac{1}{2}d_{\alpha,+}^{(1)}\cdot d_{\alpha,+}^{(m-1)},
\label{eq:dprelation}
\end{equation}
and
\begin{equation}
d_{\alpha,-}^{(m)}=d_{\alpha,-}^{(1)}\cdot
\sum_{k=0,k\equiv m+1\!\!\!\!\!\!\pmod{2}}^{\frac{m}{2}}
d_{\alpha,+}^{(k)}.
\label{eq:dmrelation}
\end{equation}
Relations \eqref{eq:dprelation} and \eqref{eq:dmrelation} reduce the multiplicative structure of $\CC[\alpha^{\frac{1}{2}},\alpha^{-\frac{1}{2}}]$ to the two elements $d_{\alpha,\pm}^{(1)}$. Consider the space
$$
Y_{\alpha}^{(n)}:=
\sum_{k=0}^n
\CC\cdot y_{\alpha}^{(k)}+
\CC\cdot d_{\alpha,+}\cdot y_{\alpha}^{(k)}.
$$
Relation \eqref{eq:dnpreimage} shows that
$$
d_{\alpha,-}\cdot Y_{\alpha}^{(n)}
=
Y_{\alpha}^{(n-1)}.
$$
Analogously we are going to prove
\begin{proposition}\label{prop:alpharelation}
For any $n>0$ we have
\begin{equation}
(\alpha^{-1}+\alpha)\cdot y_{\alpha}^{(n)}=
2 y_{\alpha}^{(n)}+y_{\alpha}^{(n-2)}.
\end{equation}
\end{proposition}

\begin{proof}
Observe that we have the relation
$$
\alpha^{-1}\cdot s_{\alpha}=
\alpha^{-\frac{1}{2}}\sum_{k=0}^\infty \alpha^k=
\alpha^{-\frac{1}{2}}+
s_{\alpha}.
$$
Therefore
$$
(\alpha^{-1}+\alpha^{1})\cdot s_{\alpha}=
\alpha^{-\frac{1}{2}}-
\alpha^{\frac{1}{2}}+
2s_{\alpha}=d_{\alpha,-}+2s_{\alpha},
$$
and
$$
(\alpha^{-1}+\alpha^{1})\cdot w_\alpha s_{\alpha}=
w_\alpha\left((\alpha^{-1}+\alpha^{1})\cdot s_{\alpha}\right)=
-d_{\alpha,-}+ 2w_\alpha s_{\alpha}.
$$
Hence
$$
(\alpha^{-1}+\alpha)\cdot y_{\alpha}^{(n)}=
(\alpha^{-1}+\alpha^{1})\cdot
(s_{\alpha,-}^n+(-1)^{n+1}w_\alpha s_{\alpha}^n)=
$$
$$
(d_{\alpha,-}+2 s_{\alpha})s_{\alpha}^{n-1}+
((-1)^{n}d_{\alpha,-}+ (-1)^{n+1}2w_\alpha s_{\alpha})w_\alpha s_{\alpha}^{n-1}
=
$$
$$
d_{\alpha,-}y_{\alpha}^{(n-1)}+2y_{\alpha}^{(n)}=
y_{\alpha}^{(n-2)}+2 y_{\alpha}^{(n)},
$$
thanks to relation \eqref{eq:dnpreimage}.
\end{proof}

We denote by $\CC[[\Lambda^{\frac{1}{2}}]]_{\alpha^{\frac{1}{2}}=1}$ a system of representatives for
$$
\CC[[\Lambda^{\frac{1}{2}}]]_{(\alpha)}/\langle\alpha^{\frac{1}{2}}-1\rangle.
$$
Then
\begin{corollary}\label{cor:kernelunibasis}
For any $n\geq 0$ the kernel of $t_\alpha^n$ is
$$
\CC[[\Lambda^{\frac{1}{2}}]]_{(\alpha)}\cdot y_{\alpha}^{(n)}=
\sum_{k=0}^n
\CC[[\Lambda^{\frac{1}{2}}]]_{\alpha^{\frac{1}{2}}=1}\cdot y_{\alpha}^{(k)}
\;+\;
\CC[[\Lambda^{\frac{1}{2}}]]_{\alpha^{\frac{1}{2}}=1}\cdot d_{\alpha,+}\cdot y_{\alpha}^{(k)}.
$$
\end{corollary}

\begin{proof}
As
$$
d_{\alpha,+}^2=2+(\alpha^{-1}+\alpha),
$$
Proposition \ref{prop:alpharelation} shows in particular that $Y_{\alpha}^{(n)}$ is a $\CC[\alpha^{\frac{1}{2}},\alpha^{-\frac{1}{2}}]$-module and the factor module
$$
Y_{\alpha}^{(n)}/Y_{\alpha}^{(n-1)}.
$$
is a vector space of dimension $2$, spanned by the classes of
$$
y_{\alpha}^{(n)},d_{\alpha,+}y_{\alpha}^{(n)}.
$$
Now the relation
$$
2\alpha^{\frac{1}{2}}=d_{\alpha,+}-d_{\alpha,-}
$$
shows that
$$
2\alpha^{\frac{1}{2}}\cdot y_{\alpha}^{(n)}=
d_{\alpha,+}\cdot y_{\alpha}^{(n)}-y_{\alpha}^{(n-1)},
$$
and a similar relation holds for $2\alpha^{-\frac{1}{2}}$, which shows that we may find a representation as claimed in the corollary.
\end{proof}

In the sequel we always assume that each element
$$
x=\sum_{\lambda}x_\lambda\lambda
\;\in\;
\CC[[\Lambda^{\frac{1}{2}}]]_{\alpha^{\frac{1}{2}}=1}
$$
in our system of representatives satisfies the boundedness condition
\begin{equation}
0\leq
\frac{\langle\lambda-\lambda_0,\alpha\rangle}{\langle\alpha,\alpha\rangle}
<\frac{1}{2}
\label{eq:repcondition}
\end{equation}
whenever $x_\lambda\neq 0$ with respect to a fixed $\lambda_0\in\Lambda^{\frac{1}{2}}$.

In our next step we generalize the results so far obtained for the univariate case to the multivariate situation. Let $\alpha_1,\dots,\alpha_r\in\Delta(\lieu,\liel')$ be pairwise distinct, and choose positive integers $n_1,\dots,n_r$, and write
$$
t_{\underline{\alpha}}^{\underline{n}}:=
t_{\alpha_1}^{n_1}\circ t_{\alpha_2}^{n_2}\circ\cdots\circ t_{\alpha_r}^{n_r},
$$
where $\underline{\alpha}=(\alpha_1,\dots,\alpha_r)$ and similarly $\underline{n}$ is the corresponding tuple of integers. The multiplication maps $t_{\alpha_i}$ commute, and we let for any $1\leq i\leq r$ denote by $\underline{\alpha}^{\{i\}}$ and $\underline{n}^{\{i\}}$ the tuples where the $i$-th component has been removed. More generally for any subset $I\subseteq\{1,\dots,r\}$ we denote by $\underline{\alpha}^I$ and $\underline{n}^I$ the resulting tuples where the entries with indices occuring in $I$ are deleted. Introduce the elements
$$
y_{\underline{\alpha}}^{\underline{n}}:=
s_{\alpha_1}^{n_1}\cdot s_{\alpha_2}^{n_2}\cdots s_{\alpha_r}^{n_r}
+
(-1)^{1+n_1+\cdots+n_r}
(w_{\alpha_1}s_{\alpha_1})^{n_1}\cdot (w_{\alpha_2}s_{\alpha_2})^{n_2}\cdots (w_{\alpha_r}s_{\alpha_r})^{n_r}
$$

\begin{proposition}\label{prop:kernelmultibasis}
For any $\underline{\alpha}$ and any $\underline{n}$ we have
$$
\kernel t_{\underline{\alpha}}^{\underline{n}}=
\CC[[\Lambda^{\frac{1}{2}}]]_{\underline{\alpha}}\cdot 
y_{\underline{\alpha}}^{\underline{n}}
\;+\;
\sum_{i=1}^r\kernel t_{\underline{\alpha}^{\{i\}}}^{\underline{n}^{\{i\}}}.
$$
\end{proposition}

We let $\CC[[\Lambda^{\frac{1}{2}}]]_{\underline{\alpha}^I=1}$ denote a system of representatives for
$$
\CC[[\Lambda^{\frac{1}{2}}]]_{\underline{\alpha}^I}/
\langle\alpha_i^{\frac{1}{2}}-1|i\not\in I\rangle,
$$
not necessarily subject to condition \eqref{eq:repcondition}. As before we have

\begin{corollary}\label{cor:kernelmultibasis}
The kernel of $t_{\underline{\alpha}}^{\underline{n}}$ is given by
$$
\sum_{I\subsetneq\{1,\dots,r\}}
\sum_{J\subseteq \overline{I}}
\CC[[\Lambda^{\frac{1}{2}}]]_{\underline{\alpha}^I=1}\cdot 
y_{\underline{\alpha}^I}^{\underline{n}^I}\cdot
\prod_{j\in J}d_{\alpha_j,+},
$$
where $\overline{I}$ denotes the complement of $I$.
\end{corollary}

We remark that by taking $G_\Lambda$-invariants, we immediately get a description of the kernel restricted to $\CC[[\Lambda]]$. Furthermore it is not hard to take the action of the Weyl group into account as well. However we will not need to make this more precise here.

\begin{proof}
We prove Proposition \ref{prop:kernelmultibasis} by induction on $r$, and also make use of the corollary in the cases guaranteed by the induction hypothesis. Consequently we assume the claim to be true for all $r'<r$ and let $z_0\in\kernel t_{\underline{\alpha}}^{\underline{n}}$.  Then
$$
d_{\alpha_1,-}^{n_1}\cdot z_0
\;\in\;
\kernel t_{\underline{\alpha}^{\{1\}}}^{\underline{n}^{\{1\}}}.
$$
For the sake of readability we set for any $J\subseteq I\subseteq\{1,\dots,r\}$, $\emptyset\subsetneq I$,
$$
y_{\underline{\alpha},I}^{\underline{n},J}:=
y_{\underline{\alpha}^{\overline{I}}}^{\underline{n}^{\overline{I}}}\cdot
\prod_{j\in J}d_{\alpha_j,+}.
$$
By the induction hypothesis we find for any $J\subseteq I\subsetneq\{2,\dots,r\}$, $I\neq\emptyset$, elements
$$
c_{I,1}^J\in\CC[[\Lambda^{\frac{1}{2}}]]_{\underline{\alpha}^{\overline{I}}=1}
$$
with the property that
$$
d_{\alpha_1,-}^{n_1}\cdot z_0
\;=\;
\sum_{\emptyset\subsetneq I\subseteq\{2,\dots,r\}}
\sum_{J\subseteq I}
c_{I,1}^J\cdot 
y_{\underline{\alpha},I}^{\underline{n},J}.
$$
By the established relations we may assume that our chosen system of representatives satisfies condition \eqref{eq:repcondition} for any $i\not\in I$ for $\lambda_0=0$. We set
$$
\tilde{z}_1:=
\sum_{\emptyset\subsetneq I\subsetneq\{2,\dots,r\}}
\sum_{J\subseteq I}
p_{\alpha_1}^{n_1}(c_{I,1}^J\cdot 
y_{\underline{\alpha},I}^{\underline{n},J}),
$$
and
$$
z_1:=
\sum_{J\subseteq \{2,\dots,r\}}
p_{\alpha_1}^{n_1}(c_{\{2,\dots,r\},1}^J\cdot 
y_{\underline{\alpha},\{2,\dots,r\}}^{\underline{n},J}).
$$
Then $z_0-z_1-\tilde{z}_1$ is in the kernel of $t_{\alpha_1}^{n_1}$, therefore we find $a_1,b_1\in\CC[[\Lambda^{\frac{1}{2}}]]_{\alpha_1=1}$ with
$$
z_0 = z_1\;+\;\tilde{z}_1 \;+\; b_1\cdot y_{\alpha_1}^{(n_1)} \;+\; a_1\cdot y_{\alpha_1}^{(n_1)}\cdot d_{\alpha_1,+}.
$$
Furthermore we note that the induction hypothesis applies to each summand of $\tilde{z}_1$, which means that we may assume that the elements $c_{I,1}^J$ are $\alpha_1$-finite for $I\subsetneq\{2,\dots,r\}$. We need to show that the same holds for the remaining cases, i.e.\ the summands of $z_1$. This will allow us to replace the section $p_\alpha^{n_1}$ by a mere multiplication, as then the sum of the elements
$$
c_{I,1}^J\cdot y_{\underline{\alpha},I\cup \{1\}}^{\underline{n},J}
$$
is a preimage of $d_{\alpha_1,-}^{n_1}\cdot z_0$ under $t_{\alpha_1}^{n_1}$, and adapting $a_1$, $b_1$, and the remaining $c_{I,1}^J$ again, we eventually find a representation as claimed.

If $r=1$ we are done, as then the elements under consideration are $0$, and in particular $\alpha_1$-finite. So we may assume that $r>1$.

If for some $1<i\leq r$ we have $\langle\alpha_1,\alpha_i\rangle\neq 0$
 we are done as well, as condition \eqref{eq:repcondition} for $\alpha_i$ implies that the coefficients in question are $\alpha_1$-finite as well.

This reduces us to the case that the roots $\alpha_1,\dots,\alpha_r$ are all pairwise orthogonal, because if there exists a pair of non-orthogonal roots, we may label one of them as $\alpha_1$, and proceed as before.

This pairwise orthogonality implies that the sections $p_{\alpha_i}$ commute with multiplications by elements whose monomials are supported only on $\alpha_j$ for $j\neq i$.

We define 
$$
d_{\alpha_2,-}^{n_2}\cdot z_{1}
\;=\;
\sum_{\emptyset\subsetneq I\subseteq\{1,3,\dots,r\}}
\sum_{J\subseteq I}
c_{I,2}^J\cdot 
y_{\underline{\alpha},I}^{\underline{n},J},
$$
with corresponding representatives $c_{I,2}^J$ as guaranteed by the induction hypothesis, and
$$
\tilde{z}_2:=
\sum_{\emptyset\subsetneq I\subsetneq\{1,3,\dots,r\}}
\sum_{J\subseteq I}
p_{\alpha_2}^{n_2}(c_{I,2}^J\cdot 
y_{\underline{\alpha},I}^{\underline{n},J}).
$$
and
$$
z_2:=
\sum_{J\subseteq \{1,3,\dots,r\}}
p_{\alpha_2}^{n_2}(c_{\{1,3,\dots,r\},2}^J\cdot 
y_{\underline{\alpha},\{1,3,\dots,r\}}^{\underline{n},J}).
$$
Then again we have representatives $a_2,b_2$ with
$$
z_{1} = z_2 \;+\; \tilde{z}_2 \;+\; 
b_2\cdot y_{\alpha_2}^{(n_2)} \;+\; a_2\cdot y_{\alpha_2}^{(n_2)}\cdot d_{\alpha_2,+},
$$
where again the induction hypothesis applies to all summands of $\tilde{z}_2$ which allows us to assume the $\alpha_2$-finiteness of the corresponding coefficients. Consequently we may assume that for $2\in I\subsetneq\{1,\dots,r\}$ and any $J\subseteq I$ there is an element
$$
\tilde{c}_{I,2}^J\in
\CC[[\Lambda^{\frac{1}{2}}]]_{\underline{\alpha}^{\overline{I}}=1}
$$
with the property that
$$
\tilde{z}_2+b_2\cdot y_{\alpha_2}^{(n_2)} + a_2\cdot y_{\alpha_2}^{(n_2)}\cdot d_{\alpha_2,+}
=
\sum_{2\in I\subsetneq\{1,\dots,r\}}
\sum_{J\subseteq I}
\tilde{c}_{I,2}^J\cdot 
y_{\underline{\alpha},I}^{\underline{n},J}.
$$
This, together with the orthogonality relation, yields the identity
$$
z_2=
\sum_{J\subseteq \{2,\dots,r\}}
p_{\alpha_{1}}^{n_{1}}(c_{\{2,\dots,r\},1}^J)\cdot 
y_{\underline{\alpha},\{2,\dots,r\}}^{\underline{n},J}
$$
$$
-
\sum_{2\in I\subsetneq\{1,\dots,r\}}
\sum_{J\subseteq I}
\tilde{c}_{I,2}^J\cdot 
y_{\underline{\alpha},I}^{\underline{n},J},
$$
where all coefficients on the right hand side are $\alpha_2$-finite, and even satisfy the stronger condition \eqref{eq:repcondition}. This implies that the coefficients occuring in $z_i$ are $\alpha_i$-finite as well. The argument goes as follows. Multiplying the above identity with
$$
d:=\prod_{j\neq 2}d_{j,-}^{n_j},
$$
and substituting the definition of $z_2$, we get
$$
\sum_{J\subseteq \{1,3\dots,r\}}
p_{\alpha_2}^{n_2}(c_{\{1,3,\dots,r\},2}^J)\cdot 
\prod_{2\neq j}d_{\alpha_j,-}^{n_j}\cdot
\prod_{j\in J}d_{\alpha_j,+}=
$$
$$
\sum_{J\subseteq \{2,\dots,r\}}
c_{\{2,\dots,r\},1}^J\cdot 
\prod_{1\neq j\neq 2}d_{\alpha_j,-}^{n_j}\cdot \prod_{j\in J}d_{\alpha_j,+}\cdot
y_{\alpha_2}^{(n_2)}
$$
$$
-
\sum_{2\in I\subsetneq\{1,\dots,r\}}
\sum_{J\subseteq I}
\tilde{c}_{I,2}^J\cdot 
\prod_{2\neq j\not\in I}d_{\alpha_j,-}^{n_j}\cdot \prod_{j\in J}d_{\alpha_j,+}\cdot
y_{\alpha_2}^{(n_2)}.
$$
By the very definition of the section $p_{\alpha_2}$, and the fact that by \eqref{eq:repcondition} the coefficients $c_{\{1,3,\dots,r\},2}^J$ are uniquely determined by the left hand side of this equation, this shows the $\alpha_2$-finiteness of every $c_{\{1,3,\dots,r\},2}^J$, concluding the proof.
\end{proof}

\begin{proposition}\label{prop:yexplicit}
For any $n\geq 1$, $\alpha\in\Delta(\lieu,\liel')$, we have
$$
s_{\alpha}^{n}=
\sum_{k=0}^\infty
\binom{n-1+k}{n-1}\cdot
\alpha^{\frac{n}{2}+k},
$$
and
$$
d_{\alpha,+}\cdot s_{\alpha}^{n}=
\sum_{k=0}^\infty
\frac{n-1+2k}{n-1+k}
\cdot
\binom{n-1+k}{n-1}\cdot
\alpha^{\frac{n-1}{2}+k},
$$
subject to the convention that for $n=1$ and $k=0$
$$
\frac{n-1+2k}{n-1+k}=1.
$$
\end{proposition}

\begin{proof}
We omit the proof, a straightforward induction on $n$.
\end{proof}

\begin{theorem}\label{thm:vanishing}
Let
$$
z=\sum_{\lambda}z_\lambda\cdot\lambda
\;\in\;
\kernel t_{\underline{\alpha}}^{\underline{n}}
$$
with the property that there exists a $\lambda_0\in\Lambda^{\frac{1}{2}}$ such that for any $1\leq i\leq r$ and any $\lambda\in\Lambda^{\frac{1}{2}}$ with
\begin{equation}
\left|
\langle \lambda-\lambda_0,\alpha_i\rangle
\right|
<
\frac{n_i+1}{2}\cdot\langle \alpha_i,\alpha_i\rangle
+
\sum_{j\neq i}\frac{n_j}{2}\cdot\langle \alpha_i,\alpha_j\rangle,
\label{regularity}
\end{equation}
we have
$$
z_\lambda = 0
$$
then
$$
z=0.
$$
\end{theorem}

\begin{proof}
We may assume that $z$ is {\em primitive}, i.e.\ does not lie in any of the kernels $\kernel t_{\underline{\alpha}^{\{i\}}}^{\underline{n}^{\{i\}}}$ for $1\leq i\leq r$. Now for any $i$ we consider
$$
z^{(i)}:=t_{\underline{\alpha}^{\{i\}}}^{\underline{n}^{\{i\}}}(z),
$$
which is primitive in $\kernel t_{\alpha_i}^{n_i}$, and hence of the form
$$
z^{(i)}
\;=\;
\sum_{k=1}^{n_i}
a_k
\cdot y_{\alpha}^{(k)}
\;+\;
b_k
\cdot d_{\alpha,+}\cdot y_{\alpha}^{(k)}
$$
by Corollary \ref{cor:kernelunibasis}, with
$$
a_1,b_1,\dots,a_{n_i},b_{n_i}\in \CC[[\Lambda^{\frac{1}{2}}]]_{\alpha^{\frac{1}{2}}=1},
$$
where we may assume that our system of representatives satisfies condition \eqref{eq:repcondition} for $\lambda_0$ as in the statement of the theorem. By Proposition \ref{prop:yexplicit} and our choice of representatives we know that if we write
$$
z^{(i)}=\sum_\mu z_\mu^{(i)}\cdot\lambda
$$
with $z_\mu^{(i)}\in\CC$, then $z_\mu^{(i)}\neq 0$ for some $\mu$ satisfying
$$
\left|
\langle \mu-\lambda_0,\alpha_i\rangle
\right|
\;<\;
\frac{n_i+1}{2}\cdot\langle \alpha_i,\alpha_i\rangle.
$$
This in turn implies the existence of a $\lambda$ with $z_\lambda\neq 0$ subject to the condition
$$
\left|
\langle \lambda-\lambda_0,\alpha_i\rangle
\right|
\;<\;
\frac{n_i+1}{2}\cdot\langle \alpha_i,\alpha_i\rangle+
\sum_{j\neq i}\frac{n_j}{2}\cdot\langle \alpha_i,\alpha_j\rangle,
$$
concluding the proof.
\end{proof}

\bibliographystyle{plain}

$$
\underline{\;\;\;\;\;\;\;\;\;\;\;\;\;\;\;\;\;\;\;\;\;\;\;\;\;\;\;\;\;\;}
$$\ \\
Karlsruher Institut f\"ur Technologie, Fakult\"at f\"ur Mathematik, Institut f\"ur Algebra und Geometrie, Kaiserstra\ss{}e 89-93, 76133 Karlsruhe, Germany\\
{januszewski@kit.edu}

\end{document}